\documentclass[12pt]{amsart}
\usepackage[utf8]{inputenc}

\usepackage{amsthm}
\usepackage{amsmath}
\usepackage{color,comment}

\usepackage{tikz}									
\usetikzlibrary{matrix}
\usetikzlibrary{patterns}
\usetikzlibrary{matrix}
\usetikzlibrary{positioning}
\usetikzlibrary{decorations.pathmorphing}
\usetikzlibrary{cd}

\usepackage[all]{xy}
\usepackage[margin=1in]{geometry}
\usepackage{verbatim}
\usepackage{mathrsfs}
\usepackage{enumitem}
\usepackage[colorlinks=true,allcolors=purple]{hyperref}
\usepackage{amssymb}
\usepackage{mathtools}

\usetikzlibrary{decorations.markings}
\usetikzlibrary{decorations.pathreplacing}
\usetikzlibrary{positioning}
\usetikzlibrary{arrows,calc,patterns}
\usetikzlibrary{intersections,through}

\usepackage{pgfplots}\pgfplotsset{compat=1.13}

\theoremstyle{plain}

\newtheorem{theorem}{Theorem}
\newtheorem*{maintheorem}{Main Theorem}
\newtheorem{proposition}[theorem]{Proposition}
\newtheorem{corollary}{Corollary}
\newtheorem{lemma}[theorem]{Lemma}

\theoremstyle{definition}
\newtheorem{definition}[theorem]{Definition}
\newtheorem{example}[theorem]{Example}

\theoremstyle{remark}

\newtheorem{remark}[theorem]{Remark}

\numberwithin{theorem}{section}
\numberwithin{equation}{section}

\def\ctM{{}^{\mathrm{ct}}\kern-2pt\mathcal M}
\def\M{\overline{\mathcal M}}
\def\Mtrop{{\mathcal M}^{\rm trop}}
\def\Mps{{\mathcal M}^{\rm ps}}
\def\Q{\mathbb Q}
\def\trop{{\rm trop}}

\def\ps{{\rm ps}}
\def\R{\mathbb{R}}
\def\Aut{{ \rm Aut}}
\def\A{\mathcal{A}}
\def\PP#1{{\rm PP}(#1)}
\def\log{{\rm log}}
\def\Spec{{\rm Spec\,}}
\def\N{\mathbb{N}}

\begin{document}

\date{}
\title{Tropical Pseudostable Curves}\

\author{Renzo Cavalieri}
\address[Cavalieri]{Colorado State University\\
Fort Collins, Colorado 80523-1874\\ USA}
\email{renzo@colostate.edu}

\author{Steffen Marcus}
\address[Marcus]{The College of New Jersey\\
Ewing, NJ 08628\\ USA}
\email{marcuss@tcnj.edu}

\begin{abstract}
We study the tropical version of the contraction morphism $\mathcal{T}$ between moduli spaces of stable and pseudostable curves. By promoting $\mathcal{T}$ to a logarithmic morphism, we obtain a piecewise linear function between the generalized cone complexes parameterizing tropical stable and pseudostable curves. The ray corresponding to the contracted divisor $\delta_1$ is not contracted to the cone point but mapped  onto a ray of $\mathcal{M}_{g,n}^{\trop, \ps}$, with a slope reflecting the geometry of the desingularization of a plane cusp. We explore in detail the situation of $g=1$, where the tautological geometry of both spaces is fully described by piecewise polynomial functions on the tropical moduli spaces.
\end{abstract}

\maketitle

\section{Introduction}

The goal of this manuscript is to introduce techniques from tropical and logarithmic geometry to the study of the tautological intersection theory of moduli spaces of pseudostable curves.
The moduli space $\Mps_{g,n}$ of \emph{pseudostable} curves, constructed by Schubert \cite{Schubert} in 1991, provides a compactification alternative to $\M_{g,n}$, wherein cusp singularities take the place of elliptic tails. This compactification admits a divisorial contraction 
\[\mathcal{T}:\M_{g,n}\longrightarrow \Mps_{g,n}\] and was shown to be a log canonical model of $\M_{g,n}$ \cite{HH}. The morphism $\mathcal{T}$ is an isomorphism away from the elliptic tails locus $\delta_1 \subset \M_{g,n}$, contracting this divisor to the codimension two locus of cuspidal curves lying completely inside the locus $\delta_0^{\rm ps} \subset \Mps_{g,n}$ consisting of pseudostable curves with a self intersecting irreducible component. 
The central challenge in using tropical methods to study the moduli space of pseudostable curves lies in the fact that cuspidal curves are not semi-stable, and thus the moduli space lacks an intrinsic tropicalization, or even an obvious choice of a well-suited logarithmic structure for logarithmic tropicalization \cite{U}. The  approach we propose here is to use the contraction map $\mathcal{T}$ to produce a reasonable tropicalization of $\Mps_{g,n}$.
We use the direct image of the logarithmic structure of $\M_{g,n}$ along $\mathcal{T}$ as the logarithmic structure on $\Mps_{g,n}$, making $\Mps_{g,n}$ a fine, saturated, logarithmically smooth logarithmic stack and $\mathcal{T}$ a logarithmic morphism. Logarithmic tropicalization \cite{U} produces a tropicalization $\mathcal{M}_{g,n}^{\trop, \rm ps}$ and, additionally,  a tropicalization of the map $\mathcal{T}$
    \[ 
    \trop(\mathcal{T}): \mathcal{M}_{g,n}^{\trop} \to \mathcal{M}_{g,n}^{\trop, \rm ps}.
    \] 
This notation reflects the following  conceptual step: by  $\mathcal{M}_{g,n}^{\trop, \rm ps}$ we denote a subcomplex of $\mathcal{M}_{g,n}^{\trop}$ which is naturally identified  with the Berkovich skeleton $\overline{\Sigma}_{g,n}^{\rm ps}$ produced by logarithmic tropicalization;  the objects parameterized by this subcomplex are defined to be pseudostable tropical curves. Our main theorem  describes $\trop(\mathcal{T})$ as a morphism of extended cone complexes.
\begin{maintheorem}
The map $\trop(\mathcal{T})$ is the unique piecewise-linear function that sends the ray dual to $\delta_1$ in $\mathcal{M}_{g,n}^{\trop}$ onto the ray dual to $\delta_0^{\rm ps}$ in $\mathcal{M}_{g,n}^{\trop, \rm ps}$ with slope 12 and sends all other extremal rays identically to their tropical pseudostable counterpart.
\end{maintheorem}

A more technical statement of this theorem is given in Theorem~\ref{thm:PLtropmap}. The factor of 12 is the central feature of the result. It is a tropical realization of the contraction of the elliptic tails divisor.
The \emph{cuspidal locus} of the moduli space of pseudostable curves produces a distinctive codimension two Chow class $\xi$. 
As an application of the Main Theorem we compute this class in terms of the irreducible divisor $\delta_0^{\rm ps}$ and the  pushforward of $\delta_0$ for stable curves:
\begin{equation}\label{eq:cusp}
    \xi = \frac{(\delta_0^\ps)^2 - \mathcal{T}_\ast(\delta_0^2) }{24}.
\end{equation}
The correspondence between piecewise-polynomial functions on the tropical moduli spaces and the subring of normally decorated strata classes on $\M_{g,n}$ and $\mathcal{M}_{g,n}^{\ps}$ utilized in the proof of Theorem~\ref{thm:PLtropmap} allows for a tropical reformulation of this equation in terms of piecewise quadratic polynomials and their pushforwards through $\mathcal{T}$ (see Equation~\ref{eq:cuspp}).
This perspective allows for a refined interpretation of \eqref{eq:cusp}: the self-intersection of the irreducible divisor may be separated in a transversal part (described by square free piecewise polynomials) and a non-transversal part (corresponding to power sums), and it is the difference of the latter part that contributes to the description of the class of the cuspidal locus.

Notice that the class of the cuspidal locus is in general not given by a piecewise polynomial, on $\mathcal{M}_{g,n}^{\trop, \ps}$,  but it is expressed as a linear combination of piecewise polynomials and their pushforwards. This is analogous to how Mumford-Arbarello-Cornalba's $\kappa$ classes on moduli spaces of curves are not in the subring generated by strata classes of any given moduli space of curves, but arise in the image of such subrings via pushforwards of forgetful morphisms. 

The case of genus $1$ is special, and we study it in detail as an application of the general theory developed. In this case, every tautological class is boundary, and it can therefore be expressed as  a piecewise polynomial on the moduli space of tropical curves. 
 The standard computation $\delta_0^2 = 0$ on $\overline{\mathcal{M}}_{1,n}$  shows that  \eqref{eq:cusp} is equivalent to a piecewise polynomial class on $\mathcal{M}_{1,n}^{\rm ps}$ (see Corollary~\ref{cor:g1poly}).

The case $(g,n) = (1,2)$ allows for some very explicit intersection theoretic computations: the projection formula gives rise to the 
coefficient of $12$ in the Main Theorem, and allows for a geometric interpretation of this coefficient in terms of the resolution of a plane cusp.
As a further application of the Main Theorem, in Theorem \ref{thm:lala} we  compute the pull-back of the class $\lambda_1^\ps$ via the contraction morphism, thus recovering in this context the result of \cite{CGRVW}. Finally we compare the morphism $\mathcal{T}$ with the analogous morphism to a moduli space of Hassett weighted curves. We find it informative to show that when $n = 2$ even though the cone complexes are identical, the piecewise linear maps differ.
\vspace{0.5cm}

\subsection{Context and Motivation}

Efforts to study the geometry of the Deligne-Mumford \cite{DM} compactification $\M_{g,n}$ of the moduli space of 
curves 
have found success in recent years by combining combinatorial techniques from logarithmic and tropical geometry with traditional intersection theoretic approaches.

In tropical geometry the moduli space of stable tropical curves $\M_{g,n}^{\rm trop}$ has 
become a central object of study \cite{M06,M06_2, GKM, GM, CV10, BMV, Chan, Viv}. It can be described combinatorially as a generalized extended cone complex that is naturally identified with the Berkovich Skeleton of $\M_{g,n}$ \cite{ACP}. Utilizing logarithmic structures as developed in the foundational work of Kazuya Kato \cite{KatoK} and Fumiharu Kato \cite{KatoF}, the paper \cite{CCUW} employs logarithmic tropicalization \cite{U} to enmesh this tropical structure on the toroidal boundary complex of the moduli space of curves with the scheme theoretic combinatorial data of its logarithmic structure. Most recently, computational approaches to studying Chow classes on the moduli space of curves have developed by employing the theory of Artin Fans \cite{AW,ACMUW} to help take advantage of the tropical/logarithmic connection. The Artin Fan $\mathcal{A}_X$ of a logarithmic algebraic stack $X$ can be viewed as a realization in the algebraic category of the generalized extended cone complex $\overline\Sigma_X$ associated to $X$. In \cite{MR:blowup} it is shown that the operational Chow ring of $\mathcal{A}_X$ is canonically isomorphic to the ring of piecewise-polynomial functions on $X$, and it comes with a map to $A^\ast(X)$. Molcho--Pandharipande--Schmitt \cite{MPS2021} use this map to define the \emph{logarithmic tautological ring} of $\M_{g,n}$, i.e., the ring of classes corresponding to piecewise-polynomial functions on the tropicalization, as a subring of the tautological ring. Our method is to extend this approach to the pseudostable case.

The tautological intersection theory of the Deligne-Mumford  compactification is well established - and as we just saw, the boundary structure of $\M_{g,n}$ is natural from both a tropical and logarithmic perspective as it is log smooth and admits only stable degenerations. It therefore seems sensible that we try to understand the tautological intersection theory of $\mathcal{M}_{g,n}^\ps$ by pulling back classes via the contraction morphism $\mathcal{T}$. Pseudostable Hodge integrals (i.e. intersection numbers of of Lambda and Psi classes on $\Mps_{g,n}$) can be computed in terms of intersection numbers on moduli spaces of stable curves in just this way \cite{CGRVW, CWill}. Analogues of the contraction morphism in combination with tropical and logarithmic techniques have also been employed to study moduli spaces of genus 1 $m$-stable maps (see \cite{RSPW1} \cite{RSPW2}) including psudostable curves as the $m=1$ case.

\subsection*{Acknowledgments} We are grateful to Dhruv Ranganathan and Jonathan Wise for interesting discussions related to this project. R.C. acknowledges with gratitude support by the NSF, through DMS -2100962. 

\section{Background}

\subsection{Stable and pseudostable curves} \label{sec:curves}
Let $g,n\in\mathbb{N}$ with  $g>0$ and $(g,n)\neq$ $(1,0)$, $(1,1)$, or $(2,0)$. Denote by $\M_{g,n}$ the moduli stack parameterizing, up to isomorphism, flat families of complete, reduced, connected, genus $g$ algebraic curves with $n$ distinct marked smooth points that are Deligne-Mumford \emph{stable} \cite{DM}. Recall that a curve $(C,p_1,\ldots,p_n)$ is stable in this sense if the log canonical divisor $\omega_C+\sum_ip_i$ is ample. Stability  is equivalent to the following conditions on the singularities and combinatorics of the curve:
\begin{itemize}
\item $C$ has at worst nodal singularities, locally isomorphic to $x^2=y^2$, 
\item every genus $0$ component has at least three special points (meaning markings or nodes), and
\item every genus $1$ component has at least one special point.
\end{itemize}
The stack $\M_{g,n}$ is smooth, proper, and irreducible of dimension $3g-3+n$. In \cite{Schubert}, Schubert introduces  a  different moduli functor for genus $g$ curves 
which allows for nodes and cusp singularities, locally isomorphic to $x^2=y^3$ and also produces a proper moduli space. Denote by $\Mps_{g,n}$ the moduli stack parameterizing, up to isomorphism, flat families of complete, reduced, connected, arithmetic genus $g$ algebraic curves with $n$ distinct marked smooth points that are \emph{pseudostable}, meaning:
\begin{itemize}
\item $C$ has at worst either nodal or cusp singularities, 
\item every (geometric) genus $0$ component has at least three special points (markings or singular points), 
\item every genus $1$ component has at least two special points, and
\item every genus $2$ component has at least 1 special point.
\end{itemize}
This stack is also smooth, proper, and irreducible of dimension $3g-3+n$. 

Denote by $$\delta = \delta_0 \cup \bigcup_{i,I} \delta_{i,I}$$ the boundary of $\M_{g,n}$, where is $\delta_0$ is the divisor consisting of curves with a self intersecting irreducible component and, for every choice of genus $i \leq g$ and subset $I\subset [n]=\{1,\ldots,n\}$ of the markings, $\delta_{i,I}$ is the divisor of curves separating the markings $I$ and $[n]-I$ on components of genus $i$ and $g-i$ respectively. Denote by $\delta_1:=\delta_{1,\emptyset}\subset\M_{g,n}$ the boundary divisor consisting of curves with an elliptic tail. In \cite[Section~2]{CGRVW}, following \cite[Theorem~1.1]{HH}, a morphism
$$\mathcal{T}:\M_{g,n}\longrightarrow \Mps_{g,n}$$
is described that relates these two compactifications of $\M_{g,n}$ by a contraction of this divisor. Away from $\delta_1$ the map $\mathcal{T}$ is an isomorphism. Given a stable curve $(C,p_1,\ldots,p_n)\in\delta_1$ we may write $C=C_0 \cup E_1 \cup \cdots \cup E_r$ by enumerating each of the elliptic tails of the curve and setting $C_0$ for the union of the remaining components. Then $\mathcal{T}([C,p_1,\ldots,p_n])$ is the isomorphism class of the curve in $\Mps_{g,n}$ obtained by replacing each $E_i$ with a cusp singularity on $C_0$ at the point where $E_i$ and $C_0$ meet. This extends to families and defines a morphism. Apart from $\delta_1$ the images of all the boundary  divisors $\delta_0, \delta_{i,I}$ through $\mathcal{T}$ produce the boundary  divisors on $\Mps_{g,n}$ (for which we analogous notation as before with an added superscript $``\ps"$). The boundary of $\Mps_{g,n}$ is the image of $\delta$, which we denote $\delta^{\rm ps}$.

\subsection{Tropical stable curves} Denote by $\M^{\rm trop}_{g,n}$ the moduli space of genus $g$ stable extended tropical curves with $n$ markings, and by $\mathcal{M}^{\rm trop}_{g,n}\subset\M^{\rm trop}_{g,n}$ the subspace of such curves with finite edge lengths. Following \cite[Definition~3.16]{CGM}, $\M^{\rm trop}_{g,n}$ is a category fibered in groupoids over the category of tropical spaces with an underlying generalized extended polyhedral complex given by the extended toroidal boundary complex $\overline{\Sigma}_{g,n}$ of $\M_{g,n}$. Interior cones $\sigma\in\Sigma_{g,n}$ of dimension $k$ in the complex correspond to codimension-$k$ strata $\Delta_\sigma $ in the boundary. We think of the generalized complex as the coarse moduli space underlying the moduli space of extended tropical curves. In this context, an extended tropical curve $\Gamma_p = (G,\gamma,d)$ corresponding to a point $p\in \overline\sigma\in \overline \Sigma_{g,n}$ is a  graph $G$ with $n$ legs together with a non-negative integer valued weight $h(v)$ at each vertex $v$ (the {\it genus} of the vertex) and nonzero edge lengths $d(e)$ taking values in $\overline{\mathbb{R}}_{\geq0}:= \mathbb{R}_{\geq0}\cup\{\infty\}$.

\subsection{Logarithmic curves and tropicalization} For a logarithmic scheme or stack $X$ we  denote its logarithmic structure by $M_X$ with structure morphism $\varepsilon:M_X\longrightarrow \mathcal{O}_X$. In \cite{kato2020} F. Kato constructs the logarithmic moduli stack $\M_{g,n}^\log$ over the category of logarithmic schemes parameterizing \emph{logarithmic curves}. Given a logarithmic scheme $S$, a logarithmic curve over $S$ is a logarithmically smooth integral logarithmic morphism $\pi:C\longrightarrow S$ whose underlying morphism of schemes is a proper and flat family of reduced, connected curves with at worst nodal singularities. Étale locally around a geometric point $p\in C$, logarithmic curves have the following combinatorially described structure:
\begin{itemize}
\item[(i)] $C={\Spec}\mathcal{O}_S[x]$, with $M_C = \pi^\ast M_S$, when $p$ is a smooth point of $C$,
\item[(ii)] $C={\Spec}\mathcal{O}_S[x]$, with $M_C = \pi^\ast M_S\oplus \N v$ where $x=\varepsilon(v)$, when $p$ is a marked point of $C$,
\item[(iii)] $C={\Spec}\mathcal{O}_S[x,y]/(xy-t)$ for some $t\in\mathcal{O}_S$, and $$M_C=\pi^\ast M_S\oplus \N a \oplus \N b / (a+b=c)$$ for some $c\in M_S$ where $\varepsilon(a)=x, \varepsilon(b)=y$, and $\varepsilon(c)=t$. 
\end{itemize}

The boundary  of $\M_{g,n}$ determines a divisorial logarithmic structure $M_{\M_{g,n}}$ on $\M_{g,n}$ making it a smooth logarithmic stack and equating the functors of points for both $\M_{g,n}^\log$ and $\M_{g,n}$ over the category of logarithmic schemes. Moreover, this logarithimic structure on $\M_{g,n}$ determines an Artin fan $\A_{\M_{g,n}}$ which captures the same information as the generalized extended complex $\overline{\Sigma}_{g,n}$ underlying the space of tropical curves but is realized in the category of logarithmic stacks. The Artin fan admits a smooth morphism
\[
\alpha : \M_{g,n}\longrightarrow \A_{\M_{g,n}}
\]
that is strict as a logarithmic morphism, meaning the logarithmic structure on $\A_{\M_{g,n}}$ pulls back to that of $\M_{g,n}$.
Logarithmic tropicalization produces $\overline{\Sigma}_{g,n}$ from a non-archimedean analytification of the map $\alpha$, giving a tropicalization map
$$
\trop : \M_{g,n}^{\beth}\longrightarrow \overline{\Sigma}_{g,n}
$$
after applying Thuillier's \cite{T} generic fiber functor functor (see \cite[Theorem~1.3]{U}).

\subsection{Tautological classes and piecewise polynomials on $\mathcal{M}_{g,n}^\trop$}  \label{sec:TropicalCurves}
There are various natural maps among the moduli spaces $\M_{g,n}$.
\begin{itemize}
\item The forgetful maps $\pi_i:\M_{g,n+1}\longrightarrow \M_{g,n}$ forget the $i$-th marked point and stabilize. Renumbering so that $i=n+1$ exhibits the universal stable curve, which we simply denote by $\pi$.
\item The gluing maps $\M_{g,n+1}\times\M_{h,m+1}\longrightarrow \M_{g+h,n+m}$ and $\M_{g,n+2}\longrightarrow \M_{g+1,n}$ glue along identified marked points and stabilize.
\item The sections $\sigma_i:\M_{g,n}\longrightarrow \M_{g,n+1}$ exhibit the section of the universal curve giving the  $i$-th marked point.
\end{itemize}
The Chow rings $A^\ast(\M_{g,n})$ contain a distinguished collection of subrings which should be thought as  minimal with respect to being geometrically meaningful (\cite{M:toward},\cite{AC:comb}). The \emph{tautological} rings $R^\ast(\M_{g,n})\subset A^\ast(\M_{g,n})$ may be defined as the smallest system of subrings of $ \sqcup_{g,n} A^\ast(\M_{g,n})$ containing the fundamental classes and closed under pushforwards via the natural maps described above. Many standard classes are tautological, including the cotangent line bundle classes $\psi_i:= c_1(\sigma_i^\ast(\omega_{\pi}))$ and the Hodge  classes $\lambda_j:=c_j(\pi_\ast(\omega_\pi))$. 

Given any polyhedral complex $\Sigma$, a piecewise polynomial function on $\Sigma$ is defined to be a continuous function $|\Sigma|\longrightarrow \R$ such that the restriction to every cone in $\Sigma$ is polynomial \cite[Definition 2.3.1]{MR:blowup}. Denote by $\PP{\Sigma}$ the ring of piecewise polynomial functions on $\Sigma$. In toric geometry, Payne \cite{P2006} has shown that when $\Sigma_X$ is the fan of a toric variety $X$, the equivariant Chow ring of $X$ can be identified with $\PP{\Sigma_X}$. In \cite{KP2008}, again in the toric case, Katz and Payne describe a map from $\PP{\Sigma_X}$ to the set of Minkowski weights on $X$. This map is neither surjective nor injective in general, but shows how a piecewise polynomial function determines Minkowski weights, giving access to a way to construct tropical cycles. 

For a generalized complex $\Sigma_X$ associated to a smooth logarithmic stack $X$, recent results \cite{MR:blowup,MPS2021} give an isomorphism between $\PP{\Sigma_X}$ and the Chow ring of the associated Artin fan $\A_X$, leading to a generalization of the toric story above. In the case of $\M_{g,n}$, the Artin fan $\A_{\M_{g,n}}$ and its strict morphism
\[
\alpha:\M_{g,n}\longrightarrow \A_{\M_{g,n}}.
\]
determine a canonical isomorphism
\begin{equation} \label{eq:ppiso}
 A^\ast(\A_{\M_{g,n}}) \cong \PP{\Sigma_{g,n}},   
\end{equation}
as described in \cite[Theorem~14]{MPS2021} and \cite[Theorem~B]{MR:blowup}.
Pullback of Chow classes along the morphism $\alpha$ identifies a subring
\[
R^\ast(\M_{g,n},\partial \M_{g,n}):= \alpha^\ast A^\ast(\A_{\M_{g,n}})
\]
of $R^\ast(\A_{\M_{g,n}})$, called the \emph{logarithmic tautological ring} of $\M_{g,n}$ \cite[5.4]{MPS2021}. This subring coincides with the $\mathbb{Q}$ linear subspace spanned by all normally decorated strata classes, that is, defined by monodromy invariant polynomials in the first Chern classes of the normal bundles to the boundary strata (see \cite[Section~5.1]{MPS2021}).
Identifying $\Sigma_{g,n}$ with the underlying generalized cone complex of $\mathcal{M}_{g,n}^\trop$, we obtain that $\alpha^\ast$ gives a map between piecewise polynomial functions on the tropical moduli spaces and the tautological ring of the moduli spaces of curves. 
\begin{remark}
One may relax the notion of piecewise polynomial function  to include functions whose domains of polynomiality  produce a refinement  of the cone complex structure of $\mathcal{M}_{g,n}^\trop$. The Chow classes thus arising via $\alpha^\ast$ are then naturally supported on some logarithmic modification of $\M_{g,n}$ \cite{MR:blowup}. We will not be needing this more sophisticated perspective in this work.
\end{remark}

\subsection{On the calculus of strata classes and piecewise polynomials.} \label{pwpcalculus}
The isomorphism \eqref{eq:ppiso} gives a powerful algebraic framework for strata class intersections on moduli spaces of curves. We  state some notations and conventions here in hope of preventing any automorphisms-related confusion. We don't expect this paragraph to be a comprehensive treatment to this calculus, but rather a working introduction to some of its (somewhat annoying) subtleties.

A stratum in $\M_{g,n}$ is the closure of a locus of curves which are topologically equivalent. A stratum is indexed by the dual graph $\Gamma$ of the generic curve it parameterizes, and it is the image of a gluing morphism $gl_\Gamma: \prod_{v\in V(\Gamma)} \M_{g_v, n_v}\to \M_{g,n}$. The dual graph $\Gamma$ determines a map $q_\Gamma: \R_{\geq 0}^{k = |E(\Gamma)|}\to \sigma_\Gamma$, where $\sigma_\Gamma$ is a (possibly folded) cone  of the tropical moduli space. The map $q_\Gamma$ has degree equal to $|\Aut(\Gamma)|$. We denote the linear coordinates  (dual to the integral structure) of $\R_{\geq 0}^{k}$ by $x_1, \ldots, x_k$.
Denote by $\varphi_\Gamma$ the piecewise polynomial function on $\Mtrop_{g,n}$ which equals $\prod_{i=1}^k x_i$ on the cone  $\sigma_\Gamma$\footnote{Technically we are defining the function on the orthant $\R_{\geq 0}^{k}$ covering $\sigma_\Gamma$, but since the function is symmetric in all the coordinates, it descends to the quotient regardless of what $\Aut(\Gamma)$ is.}, and zero on all other cones.

In the case when the graph $\Gamma$ has no nontrivial automorphisms, we have unambiguously:
\begin{equation}
    \delta_\Gamma := {gl_\Gamma}_\ast([1]) = \alpha^\ast(\varphi_\Gamma).
\end{equation}
Also, if we denote by $e_i$ the edge of $\Gamma$ dual to the coordinate $x_i$, and by $\bullet_i, \star_i$ the two half edges that get glued to form $e_i$, we have that:
\begin{equation}
\delta_\Gamma \cdot (-\psi_{{\bullet}_i}-\psi_{\star_i}):=  {gl_\Gamma}_\ast([-\psi_{{\bullet}_i}-\psi_{\star_i}]) = \alpha^\ast(x_i\cdot \varphi_\Gamma).  
\end{equation}

When $\Gamma$ has nontrivial automorphisms, some care is needed. 
The map $gl_\Gamma$ is a finite map of degree $|{\rm Aut}(\Gamma)|$, hence when the graph has non-trivial automorphisms it is no longer an isomorphims onto its image. 

We define the {\bf stratum class} $\delta_\Gamma$ to be the pushforward of the fundamental class via the above gluing morphism:
\[
\delta_\Gamma := {gl_\Gamma}_\ast([1]).
\]
What we are calling strata class here differs by a factor of $|{\rm Aut}(\Gamma)|$ from the corresponding notion in some earlier literature.
We make this convention to align any strata computation with how they are coded in the program {\tt admcycles}, where the product of strata is computed via fiber products of the corresponding gluing morphisms.

The cone $\sigma_\Gamma$ is now a quotient of $\mathbb{R}_{\geq 0}^k$ by the group $\Aut(\Gamma)$, and we may describe piecewise polynomial functions on $\sigma_\Gamma$ as polynomials in the $x_i$'s with the appropriate symmetries to descend to the quotient. Note that the function $\varphi_\Gamma$ is symmetric in all variables and hence will always define a function on any folded cone. 

The map $\alpha^\ast$ may in principle be modified by rescaling the piecewise linear functions corresponding to boundary divisors. However there is a best choice of normalization in which $\varphi_\Gamma$ is:
\begin{equation}
\alpha^\ast(\varphi_\Gamma) = \frac{\delta_\Gamma}{|\Aut(\Gamma)|}=   \frac{{gl_\Gamma}_\ast([1])}{|\Aut(\Gamma)|}.
\end{equation}

Let $\rho$ be a ray and $\sigma$ an arbitrary cone of $\Mtrop_{g,n}$. Denote by $\Gamma_\rho, \Gamma_\sigma$ the corresponding dual graphs and by $E(\Gamma_\sigma)_\rho\subseteq E(\Gamma_\sigma)$ the subset of edges of $\Gamma_\sigma$ corresponding to the ray $\rho\subset \sigma$. Then

\begin{equation}
  \frac{{gl_{\Gamma_\sigma}}_\ast(\sum_{i\in E(\Gamma_\sigma)_\rho}(-\psi_{{\bullet}_i}-\psi_{\star_i}))}{|\Aut(\Gamma_\sigma)|} = \alpha^\ast(\Phi_\rho), 
\end{equation}
where $\Phi_\rho$ is the piecewise quadratic function that restricts to any cone $\sigma$ as follows:
\begin{equation}
  {\Phi_\rho}_{|\sigma}  =  \sum_{i\in E(\Gamma_\sigma)_\rho} x_i^2.
\end{equation}

With this perspective the self-intersection of a divisor is separated into 
a transversal part, consisting of the square free terms in the expansion of $\varphi_\rho^2$, and a non-transversal one controlled by the function $\Phi_\rho$. We conclude with a simple example that shows how everything fits together.

\begin{example}
Consider the space $\M_{1,2}$, and denote by $\Gamma$ the dual graph to the zero-dimensional stratum of banana curves. Since $\delta_0$ is pulled-back from $\M_{1,1}$, it is clear that $\delta_0^2 = 0$. Performing the computation using {\tt admcycles}, one obtains:
\begin{equation}\label{selfint}
    \delta_0^2 = 2\delta_0 (-\psi_{\bullet}-\psi_{\star})+ 4\delta_\Gamma.
\end{equation}
One can then easily check the vanishing since $\delta_0 \cdot \psi_\bullet=  \delta_0 \cdot \psi_\star = \delta_\Gamma = [pt.]$.

Denote by $x$ the linear coordinate for the ray corresponding to $\delta_0$, and by $x_1, x_2$ the coordinates for the folded cone corresponding to $\Gamma$.
On the folded cone we have:
\begin{equation}
    \alpha^\ast(x_1+x_2) = \frac{\delta_0}{2}, \ \ \ \  \alpha^\ast(x_1^2+x_2^2)= \frac{\delta_0(-\psi_{\bullet}-\psi_{\star})}{2}, \ \ \ \ \alpha^\ast(x_1x_2) = \frac{\delta_\Gamma}{2}.
\end{equation}
It immediately follows that \eqref{selfint} is four times the expansion of the square of the binomial $(x_1+x_2)$. Note that  the piecewise polynomial  $(x_1+x_2)^2 \not= 0$; this is a simple example of a non-trivial element in $\ker(\alpha^\ast)$.
\end{example}

\subsection{Pseudostable Chow} The moduli space of pseudostable curves admits a universal curve $\pi:\mathcal{C}^{\rm ps}\longrightarrow \Mps_{g,n}$ with sections $\sigma_i^{\rm ps}: \Mps_{g,n} \longrightarrow \mathcal{C}^{\rm ps}$ for each marked point. This provides for the analogous construction of cotangent line bundle classes $\psi_i:= c_1(\sigma_i^\ast(\omega_{\pi}))$ and Hodge bundle classes $\lambda_j:=c_j(\pi_\ast(\omega_\pi))$ in $A^\ast(\Mps_{g,n})$. Their relationship via $\mathcal{T}$ to the $\psi$ and $\lambda$ classes on $\M_{g,n}$ are given by the pullback formulas
\begin{align*}
    \mathcal{T}^\ast (\psi_i) = \psi_i \text{ for all } i=1,\ldots,n\\
    \mathcal{T}^\ast (\lambda_j) = \lambda_j + \sum_{i=1}^j\frac{1}{i!} \mathcal{G}_\ast^i(p_0^\ast(\lambda_{j-i}))
\end{align*}
as studied in \cite{CGRVW}, where 
\[\mathcal{G}^k:\M_{g-k,n+k}\times\M_{1,1}^{\times k} \longrightarrow \M_{g,n}\] 
is the gluing map onto the boundary locus of $k$-elliptic tails in $\M_{g,n}$, and $p_0$ is the projection onto the first factor $\M_{g-k,n+k}$. 

\section{Tautological geometry of pseudo-stable curves}

\subsection{Forgetful morphisms and universal curves}

Moduli spaces of pseudostable curves also admit forgetful morphisms, however the universal curve $\pi:\mathcal{C}^{\rm ps}\longrightarrow \Mps_{g,n}$  is not  identified with a morphism forgetting a marked point. The next Lemma describes the relationship between a forgetful morphism of moduli spaces of stable curves and its pseudostable counterpart. 

\begin{lemma}
   Let $\pi$ denote the universal curve morphism and $\pi_{n+1}$ the $(n+1)$-th forgetful morphisms (ditto with superscript $\ps$ for pseudostable curves). We have the following commutative diagram, where the map $c$ contracts a divisor to a codimension three locus. In particular, the forgetful morphism among moduli spaces of pseudo-stable curves is not flat. 
    $$
    \xymatrix{\M_{g,n+1}\ar[r]^{\mathcal{T}} \ar@{=}[d]  \ar@/_2.0pc/[dd]_{\pi_{n+1}}& \Mps_{g,n+1} \ar[d]^c \ar@/^2.0pc/[dd]^{\pi_{n+1}^\ps}\\
    \mathcal{C} \ar[d]_{\pi} \ar[r] & \mathcal{C}^{\rm ps}\ar[d]^{\pi^\ps} \\
    \M_{g,n}\ar[r]^{\mathcal{T}} & \Mps_{g,n}
    }
    $$
\end{lemma}
\begin{proof}
    The commutativity of the above diagram follows in a straightforward way from the modular description of all the morphisms involved. 
    Let $\xi \in \Mps_{g,n}$ denote a moduli point  corresponding to a pointed curve $(C, p_1, \ldots, p_n)$ with one cusp at a point $q$; then $\mathcal{T}_\xi = \mathcal{T}^{-1}(\xi)$ is a curve in $\M_{g,n}$ isomorphic to $\M_{1,1}$. Its points parameterize nodal curves  $C^\nu\cup E$, where $C^\nu$ is the   normalization of the cusp of $C$, and $E$ is an elliptic tail attached at the preimage of the cusp. The inverse image $\pi_{n+1}^{-1}(\mathcal{T}_\xi)$ is a reducible surface $S = S_1\cup_{\mathcal{T}_\xi} S_2$, where $S_1 = C^\nu \times \mathcal{T}_\xi$, $S_2\cong \M_{1,2}$ and the locus $\{q\}\times \mathcal{T}_\xi \cong \M_{1,1}$ in $S_1$ is glued with the image of the section $\sigma_1: \M_{1,1} \to \M_{1,2}$ in $S_2$. 
    We have the equality $\mathcal{T}(S) = \pi_{n+1}^{-1}(\xi) \subseteq \Mps_{g, n+1}$. The map $\mathcal{T}$ contracts the $\M_{1,1}$ direction in the surface $S_1$ and the distinguished curve $\mathcal{T}_\xi$ in $S_2$ to a point, to give:
    \[
    \pi_{n+1}^{-1}(\xi) = C \cup_q \widetilde{S}_2,
    \]
    with $\widetilde{S_2} \cong \Mps_{1,2}$. It follows immediately that the map $\pi_{n+1}: \Mps_{g, n+1}\to \Mps_{g,n}$ is not flat, as it does not have equidimensional fibers.
    The map $c$ maps all of $\widetilde{S}_2$ to $q$, the cuspidal point of $C$. This analysis is illustrated in Figure \ref{fig:fiber}.

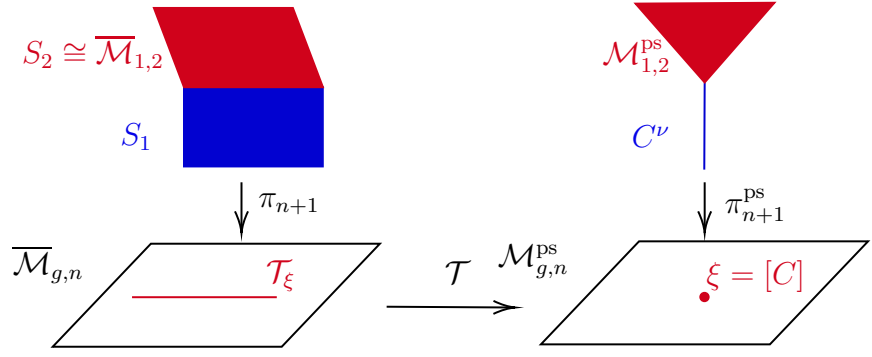
\begin{figure}
    \centering

\tikzset{every picture/.style={line width=0.75pt}} 

\begin{tikzpicture}[x=0.75pt,y=0.75pt,yscale=-1,xscale=1]

\draw   (146.5,388) -- (262,388) -- (212.5,440) -- (97,440) -- cycle ;
\draw   (392.7,386.8) -- (508.2,386.8) -- (458.7,438.8) -- (343.2,438.8) -- cycle ;
\draw [color={rgb, 255:red, 208; green, 2; blue, 27 }  ,draw opacity=1 ]   (137,415) -- (210,415) ;
\draw  [color={rgb, 255:red, 208; green, 2; blue, 27 }  ,draw opacity=1 ][fill={rgb, 255:red, 208; green, 2; blue, 27 }  ,fill opacity=1 ] (423.6,414.9) .. controls (423.6,413.74) and (424.54,412.8) .. (425.7,412.8) .. controls (426.86,412.8) and (427.8,413.74) .. (427.8,414.9) .. controls (427.8,416.06) and (426.86,417) .. (425.7,417) .. controls (424.54,417) and (423.6,416.06) .. (423.6,414.9) -- cycle ;
\draw  [color={rgb, 255:red, 2; green, 2; blue, 208 }  ,draw opacity=1 ][fill={rgb, 255:red, 2; green, 2; blue, 208 }  ,fill opacity=0.22 ] (163.2,309) -- (233.2,309) -- (233.2,349) -- (163.2,349) -- cycle ;
\draw  [color={rgb, 255:red, 208; green, 2; blue, 27 }  ,draw opacity=1 ][fill={rgb, 255:red, 208; green, 2; blue, 27 }  ,fill opacity=0.2 ] (148.03,269) -- (218.11,269) -- (233.27,309) -- (163.2,309) -- cycle ;
\draw [color={rgb, 255:red, 2; green, 2; blue, 208 }  ,draw opacity=1 ]   (425.8,306.6) -- (425.6,350.6) ;
\draw  [color={rgb, 255:red, 208; green, 2; blue, 27 }  ,draw opacity=1 ][fill={rgb, 255:red, 208; green, 2; blue, 27 }  ,fill opacity=0.2 ] (425.8,306.6) -- (390.95,266.82) -- (460.95,266.38) -- cycle ;
\draw    (265.4,419.8) -- (328.2,420.19) ;
\draw [shift={(330.2,420.2)}, rotate = 180.35] [color={rgb, 255:red, 0; green, 0; blue, 0 }  ][line width=0.75]    (10.93,-3.29) .. controls (6.95,-1.4) and (3.31,-0.3) .. (0,0) .. controls (3.31,0.3) and (6.95,1.4) .. (10.93,3.29)   ;
\draw    (192.2,357) -- (192.2,380.6) ;
\draw [shift={(192.2,382.6)}, rotate = 270] [color={rgb, 255:red, 0; green, 0; blue, 0 }  ][line width=0.75]    (10.93,-3.29) .. controls (6.95,-1.4) and (3.31,-0.3) .. (0,0) .. controls (3.31,0.3) and (6.95,1.4) .. (10.93,3.29)   ;
\draw    (425.8,357) -- (425.8,380.6) ;
\draw [shift={(425.8,382.6)}, rotate = 270] [color={rgb, 255:red, 0; green, 0; blue, 0 }  ][line width=0.75]    (10.93,-3.29) .. controls (6.95,-1.4) and (3.31,-0.3) .. (0,0) .. controls (3.31,0.3) and (6.95,1.4) .. (10.93,3.29)   ;

\draw (75.4,387.8) node [anchor=north west][inner sep=0.75pt]    {$\M_{g,n}$};
\draw (320,386.6) node [anchor=north west][inner sep=0.75pt]    {$\mathcal{M}_{g,n}^\ps$};
\draw (199.2,361) node [anchor=north west][inner sep=0.75pt]    {$\pi_{n+1}$};
\draw (293.2,395.4) node [anchor=north west][inner sep=0.75pt]    {$\mathcal{T}$};
\draw (434.4,358.6) node [anchor=north west][inner sep=0.75pt]    {$\pi_{n+1}^\ps$};
\draw (129.6,325.8) node [anchor=north west][inner sep=0.75pt]    {$\textcolor[rgb]{0.02,0.01, 0.91}{S_1}$};
\draw (80,281.4) node [anchor=north west][inner sep=0.75pt]    {$\textcolor[rgb]{0.82,0.01, 0.11}{S_2\cong \M_{1,2}}$};
\draw (387.6,327) node [anchor=north west][inner sep=0.75pt]    {$\textcolor[rgb]{0.02,0.01, 0.91}{C^\nu}$};
\draw (372.8,283.8) node [anchor=north west][inner sep=0.75pt]    {$\textcolor[rgb]{0.82,0.01, 0.11}{\mathcal{M}_{1,2}^\ps}$};
\draw (425.2,395) node [anchor=north west][inner sep=0.75pt]    {$\textcolor[rgb]{0.82,0.01,0.11}{\xi = [C]}$};
\draw (203.2,395) node [anchor=north west][inner sep=0.75pt]    {$\textcolor[rgb]{0.82,0.01,0.11}{\mathcal{T}_\xi}$};

\end{tikzpicture}

    \caption{The forgetful morphism between moduli spaces of pseudostable curves is not flat. In color we depict the fibers over a general curve in the cuspidal locus of $\mathcal{M}_{g,n}^\ps$.}
    \label{fig:fiber}
\end{figure}

    The locus of cuspidal curves is a codimension two closed subset of $\Mps_{g,n}$, and therefore it follows that the locus of cusps is codimension three in the universal family. By our fiberwise analysis, the inverse image via $c$ of every point of the locus of cusps is a surface, hence $c$ contracts a divisor.
\end{proof}

\subsection{Pseudostable curves and tropicalization}\label{subsection:trop} The morphism $\mathcal{T}:\M_{g,n}\longrightarrow \Mps_{g,n}$ allows us to extend the logarithmic and tropical approach of Section~\ref{sec:TropicalCurves} to the setting of pseudostable curves. 

Denote by $M_{\Mps_{g,n}}$ the logarithmic structure on $\Mps_{g,n}$ constructed by taking the direct image of $M_{\M_{g,n}}$ along $\mathcal{T}$. This makes $\Mps_{g,n}$ a logarithmic stack and promotes $\mathcal{T}$ to a logarithmic morphism. As described in Section~\ref{sec:curves}, $\mathcal{T}$ contracts $\delta_1$ and is an isomorphism away from this boundary divisor, so by construction $M_{\Mps_{g,n}}$ coincides with the divisorial logarithmic structure determined by the boundary $\delta^{\rm ps} := \mathcal{T}(\delta) \subset \Mps_{g,n}$. With this logarithmic structure, $\Mps_{g,n}$ is logarithmically smooth (see Proposition~\ref{prop:logsmooth}). To construct a tropicalization morphism in the trivially valued setting, we apply the general theory in \cite{U} for logarithmic tropicalization of fine, saturated logarithmic stacks as in the case of stable cuves.

Let $k$ be an algebraically closed field endowed with the trivial absolute value. Denote by $\overline{\Sigma}_{g,n}^{\rm ps}$ the generalized extended complex determined by the (toroidal) boundary divisor $\delta^{\rm ps}$. Cone by cone, this generalized complex can be constructed directly from the logarithmic structure $M_{\Mps_{g,n}}$ (see \cite[Section~6]{U}). A non-archemedian analytification ${\Mps_{g,n}}^{\beth}$ in this setting is given by applying a generalization \cite{U}[Proposition~5.3] of Thuillier’s generic fiber functor suitable for algebraic stacks of finite type over $k$. We get an induced tropicalization map. %

\begin{lemma}
    There exists a unique, natural, continuous tropicalization map
    \[ 
    {\trop_{\rm ps}} : |{\Mps_{g,n}}^{\beth}|\longrightarrow \overline{\Sigma}_{g,n}^{\rm ps}.
    \] 
    from the underlying topological space of ${\Mps_{g,n}}^{\beth}$ to the extended complex $\overline{\Sigma}_{g,n}^{\rm ps}$. 
\end{lemma}

\begin{proof}
The map $\trop_{\rm ps}$ is a direct application of \cite[Proposition~6.3(i)]{U}. A description in any local chart $X=\Spec P \rightarrow \Mps_{g,n}$ is provided in \cite[Remark 6.2]{U}.  It is the map \[X^\beth \longrightarrow \overline{\sigma}_X={\rm Hom}(P,\overline{\mathbb{R}}_{\geq0})\] given by \[x\mapsto (p\mapsto -\log|p|_X).\] A  combinatorial description is given in Section~\ref{subsec:tropicalpseudostable}.

This map can also seen arising in the following way. The logarithimic structure on $\Mps_{g,n}$ 
determines an Artin fan $\A_{\Mps_{g,n}}$ and a smooth and strict morphism of logarithmic stacks $\phi:\Mps_{g,n}\longrightarrow \A_{\Mps_{g,n}}$. Applying the analytification functor $(\cdot)^{\beth}$ to the morphism $\phi$ gives a map of strict analytic stacks
\[\phi^\beth:{\Mps_{g,n}}^\beth\longrightarrow\A_{\Mps_{g,n}}^\beth.\]
Taking the underlying topological spaces gives a continuous map
\[\phi^\beth:\left |{\Mps_{g,n}}^\beth\right| \longrightarrow \left| \A_{\Mps_{g,n}}^\beth\right|. \]
Finally, by \cite[Theorem~1.1]{U} there is a natural homeomorphism
\[\mu:\left| \A_{\Mps_{g,n}}^\beth\right| \longrightarrow \overline{\Sigma}_{g,n}^{\rm ps}.\] The composition $\phi^\beth\circ \mu$ gives the  tropicalization map.
\end{proof}

As in the stable case, the Artin fan 
\[
\alpha^{\rm ps}:\Mps_{g,n}\longrightarrow \A_{\Mps_{g,n}}
\]
determines a canonical isomorphism
\begin{equation}\label{eq:ppisops}
A^\ast(\A_{\Mps_{g,n}}) \cong \PP{\Sigma_{g,n}^{\rm ps}},
\end{equation}
and pullback along $\alpha$ identifies the logarithmic tautological ring
\[
R^\ast(\Mps_{g,n},\partial \Mps_{g,n}):= \alpha^\ast A^\ast(\A_{\Mps_{g,n}})
\]
of $\Mps_{g,n}$. This allows us to interpolate between tropical Chow classes described by piecewise polynomials on $\Sigma_{g,n}^{\rm ps}$ and normally decorated strata classes in the Chow ring of the moduli space of pseudostable curves.

The extended complex $\overline{\Sigma}_{g,n}^{\rm ps}$ plays the role of the tropicalization of the moduli space of pseudostable curves. Moreover, the functoriality of logarithmic tropicalization in \cite[Proposition~6.3]{U} allows us to tropicalize the morphisnm $\mathcal{T}$. 

\begin{lemma}
    The contraction $\mathcal{T}:\M_{g,n}\longrightarrow \Mps_{g,n}$ induces a morphism of extended generalized complexes $\trop(\mathcal{T}): \overline{\Sigma}_{g,n}\longrightarrow\overline{\Sigma}_{g,n}^{\rm ps}$ fitting into a commutative diagram
    \[
    \begin{tikzcd}
        \left| {\M_{g,n}}^{\beth} \right| \arrow[r,"\mathcal{T}^\beth"] \arrow[d,"\trop"']&  \left| {\Mps_{g,n}}^{\beth}\right| \arrow[d,"\trop_{\rm ps}"]\\
        \overline{\Sigma}_{g,n} \arrow[r, "\trop(\mathcal{T})"]& \overline{\Sigma}_{g,n}^{\rm ps}
    \end{tikzcd}
\]
of topological spaces.
\end{lemma}

\begin{proof}
    \cite[Proposition~6.3(ii)]{U}
\end{proof}

The map $\trop(\mathcal{T})$  allows us to compare these two compactifications of $\mathcal{M}_{g,n}$ tropically. Indeed, our goal is to relate normally decorated strata classes by manipulating piecewise polynomials on the respective tropical moduli spaces. As a first step we must identify the  morphism $\PP{\Sigma_{g,n}^{\rm ps}} \longrightarrow \PP{\Sigma_{g,n}}$ on piecewise polynomials which commutes with pullback of Chow classes along $\mathcal{T}$.

While the construction of Artin fans is functorial along strict logarithnic morphisms, the map $\mathcal{T}$ is not strict and some consideration must be taken \cite[Section~5.4]{ACMUW}. Fortunately, $\mathcal{T}$ is a contraction of a logarithmic divisor between stacks with locally free logarithmic structures, so this necessary functoriality is preserved and we are assured a commutative diagram
\begin{equation}\label{eq:functorial}
\xymatrix{
\M_{g,n} \ar[r]^{\mathcal{T}} \ar[d]_{\alpha}& \Mps_{g,n} \ar[d]^{\alpha^{\rm ps}}\\
\mathcal{A}_{\M_{g,n}} \ar[r] & \mathcal{A}_{\Mps_{g,n}} 
}
\end{equation}
with the bottom arrow giving a map between the respective Artin fans. The proof of this fact is due to Jonathan Wise and is provided in Appendix~\ref{appendix}. 

Pulling back Chow classes along the maps in \eqref{eq:functorial} and an application of the isomorphisms \ref{eq:ppiso} and \ref{eq:ppisops} produces the commutative diagram 
\begin{equation}\label{piecewisecommute}
\xymatrix{
 A^\ast(\Mps_{g,n}) \ar[r]^{\mathcal{T}^\ast} &  A^\ast(\M_{g,n})\ar[d]^{(\alpha^{\rm ps})^\ast}\\
\PP{\Sigma_{g,n}^{\rm ps}} \ar[u]_{\alpha^\ast} \ar[r] & \PP{\Sigma_{g,n}}
}
\end{equation} 
interpolating between piecewise polynomials on cone complexes and logarithmic tautological classes on $\Mps_{g,n}$ and $\M_{g,n}$.

\subsection{Pseudostable tropical  curves and their moduli.} \label{subsec:tropicalpseudostable}
In this section we define the notion of tropical pseudostable curve and  exhibit an extended cone complex parameterizing them which is naturally identified with $\overline{\Sigma}_{g,n}^{\rm ps}$.

\begin{definition}\label{def:psst}
A tropical curve of genus $g$ with $n$ marked points is {\bf pseudostable} if every vertex of genus zero is at least three-valent, any three-valent genus zero vertex does not have a self-edge and every vertex of genus one is at least two-valent.\\
The {\bf moduli space of (extended) pseudostable tropical curves} $\mathcal{M}^{\trop,\rm ps}_{g,n}$ ($\overline{\mathcal{M}}^{\trop,\rm ps}_{g,n}$) is the sub-complex of $\mathcal{M}^{\trop}_{g,n}$ ($\overline{\mathcal{M}}^{\trop}_{g,n}$) of points parameterizing tropical curves that are pseudostable.
\end{definition}

We give an explicit description of this subcomplex. Given a cone $\tau$ in a cone complex $\Sigma$, we define the {\it open star of $\tau$} to be the union of the relative interior of all cones containing $\sigma$:
\begin{equation}
    {\rm St}^\diamond(\tau):= \bigcup_{\sigma \succeq \tau} \sigma^\circ.
\end{equation}

Given any divisor $\delta_{i, I}$ in $\overline{\mathcal{M}}_{g,n}$ we denote by $\rho_{i,I}$ the ray in $\mathcal{M}_{g,n}^{\trop}$ parameterizing tropical curves whose topological type is the dual graph of the general point of $\delta_{i, I}$. In particular we denote by $\rho_1$   the ray dual to $\delta_1$ and $\rho_0$ to $\delta_0$.

\begin{lemma}
The moduli space of pseudostable tropical curves is the complement of the open star of $\rho_1$ in $\mathcal{M}_{g,n}^{\trop}$:
\begin{equation}
    \mathcal{M}_{g,n}^{\trop, \rm ps} = \mathcal{M}_{g,n}^{\trop} \smallsetminus St^\diamond(\rho_1),
\end{equation}
and the space of extended pseudostable tropical curves is the closure of     $\mathcal{M}_{g,n}^{\trop, \rm ps}$ in $\M_{g,n}^{\trop}$.
\end{lemma}
\begin{proof}
By definition \ref{def:psst} a tropical pseudostable curve is stable, and therefore the space of pseudostable tropical curves may be realized as a subset of $\M_{g,n}^{\trop}$.
A point in the open star of $\rho_1$ parameterizes a tropical curve which has either a one-valent vertex of genus one, or a trivalent genus zero vertex with a self-edge. These are precisely the stable tropical curves that are not pseudostable. 
\end{proof}

In identical fashion to \cite[Theorem 1.2.1, Part (1)]{ACP} one then may identify the Berkovich skeleton $\overline{\Sigma}_{g,n}^{\rm ps}$ with the moduli space of tropical pseudostable curves.

\begin{lemma} \label{lem:skelequalstrop}
   There is an isomorphism of generalized, extended cone complexes with integral structures:
   \[
   \overline{\Sigma}_{g,n}^{\rm ps}\cong   \M_{g,n}^{\trop, \rm ps}. 
   \]
\end{lemma}
\begin{proof}
The proof of this Lemma follows by observing that both spaces are obtained as colimits of a system of cones and morphisms, and there are natural bijections between the two systems.   
\end{proof}

We now turn our attention to the map $\trop(\mathcal{T})$, which we can think of as a map between moduli spaces of tropical curves.
\begin{lemma}
\label{lem:compatibility}
The following diagram is commutative:
\begin{equation} \label{eq:commdiag}
    \xymatrix{
PP(\mathcal{M}_{g,n}^{\trop, \rm ps}) \ar[r]^{\alpha^\ast_{\rm ps}} \ar[d]_{\trop(\mathcal{T})^\ast}& A^\ast(\mathcal{M}^{\rm ps}_{g,n})\ar[d]^{\mathcal{T}^\ast}\\
PP(\mathcal{M}_{g,n}^{\trop}) \ar[r]^{\alpha^\ast} & A^\ast(\M_{g,n}) 
}
\end{equation}
Further, the map $\alpha^\ast$ is injective when restricted to piecewise linear functions.
\end{lemma}

\begin{proof}
    Diagram \eqref{eq:commdiag} is a reformulation of diagram \eqref{piecewisecommute} via the identification of the skeleta of the two moduli spaces with the tropical moduli spaces from \cite{ACP} and Lemma \ref{lem:skelequalstrop}.

  For any ray  $\rho$  in $\mathcal{M}_{g,n}^{\trop}$, define $\varphi_{\rho}$ to be the piecewise linear function with slope one along that ray and zero along any other ray of the cone complex. These functions form a basis for the vector space of piecewise linear functions on $\mathcal{M}_{g,n}^{\trop}$.
  For any  ray $\rho$, ${\alpha^\ast}(\varphi_{\rho}) = D_\rho$, the boundary divisor dual to the ray $\rho$.
  The injectivity of ${\alpha^\ast}$ follows from the well known fact that there are no linear relations among boundary divisors of $\M_{g,n}$ when $g\ge 1$ (\cite[Theorems 1,2]{AC:Pic}).
\end{proof}

\begin{theorem}
\label{thm:PLtropmap}
The map 
\[
\trop(\mathcal{T}): \mathcal{M}_{g,n}^{\trop}
\to \mathcal{M}_{g,n}^{\trop, \rm ps}
\]
is the identity function on $\mathcal{M}_{g,n}^{\trop}\smallsetminus Star^\diamond(\rho_1)$. Denoting by $e_0,e_0^\ps, e_1$ the primitive generators for the rays $\rho_0, \rho_0^\ps, \rho_1$, we have:
\[
\trop(\mathcal{T})(e_1) =   12 e_0^\ps
\]
The map is then defined on $Star^\diamond(\rho_1)$ by linearly extending this assignment.
\end{theorem}
\begin{proof}

 For any ray  $\rho$ (resp. $ \rho^{\rm ps}$) in $\mathcal{M}_{g,n}^{\trop}$  (resp. $\mathcal{M}_{g,n}^{\trop, \rm ps}$), define $\varphi_{\rho}$ (resp. $\varphi_{\rho^{\rm ps}}$ ) to be the piecewise linear function with slope one along that ray and zero along any other ray of the cone complex; for the rays $\rho_0, \rho_1$, dual to the strata $\delta_0, \delta_1$, we denote the corresponding piecewise linear functions simply by $\varphi_0, \varphi_1$ (with superscripts in the pseudostable case).
   The map $\trop(\mathcal{T})$ is a piecewise-linear map of cone complexes; hence, to describe it, it suffices to describe the pull-backs  $\trop(\mathcal{T})^\ast(\varphi_{\rho^{\rm ps}})$ for all rays of $\mathcal{M}_{g,n}^{\trop, \rm ps}$. Using the commutative diagram in Lemma \ref{lem:compatibility}, one may study such pull-backs via the pull-back of boundary divisors of $\mathcal{M}_{g,n}^{\rm ps}$. Recall that $\alpha^\ast_{\ps}(\varphi_{\rho^{\rm ps}}) = \delta_{\rho^{\rm ps}}$, the boundary divisor dual to the ray. 
   For all boundary  divisors of $\mathcal{M}_{g,n}^{\rm ps}$ except $\delta_0$, it is immediate to see that $\mathcal{T}^\ast(\delta_{\rho^\ps}) = \delta_\rho$. By the injectivity of ${\alpha^\ast}$ in degree one, we then conclude that, for all rays $\rho\not= \rho_0,$
 \begin{equation} \label{eq:everywherelse}
      \trop(\mathcal{T})^\ast(\varphi_{\rho^\ps}) = \varphi_\rho.  
 \end{equation}
For the divisor $\delta_0^{\ps}$, we have that 
\[\mathcal{T}^{-1}(\delta_0^\ps) = \delta_0 \cup \delta_1,\]
from which it follows that the pull-back of $\delta_0^\ps$ is a linear combination of $\delta_0$ and $\delta_1$. One may see that the coefficient of $\delta_0$ is one, either by using the projection formula, or by observing that $\mathcal{T}$ is an isomorphism when restricted to an open dense subset of $\delta_0$, so 
\begin{equation}\label{eq:x}\mathcal{T}^\ast(\delta_0^\ps) = \delta_0 +q \delta_1,\end{equation}
for some $q\in \Q$ to be determined.

The coefficient $q$ may be computed with one intersection computation. We perform such computation for $\mathcal{T}:\M_{1,2}\to \mathcal{M}^\ps_{1,2}$ in Lemma \ref{lem:comp24} and determine $q = 24$. Here we show that the one numerical computation in $(g,n) = (1,2)$ determines the coefficient $q$ for all values of $(g,n)$.

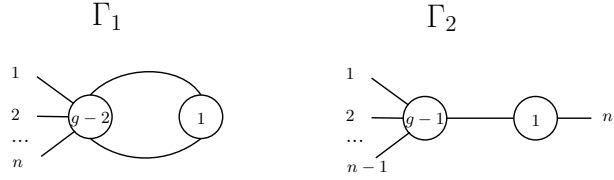
\begin{figure}[tb]
    \centering

\resizebox{.5 \textwidth}{!}{
\tikzset{every picture/.style={line width=0.75pt}} 

\begin{tikzpicture}[x=0.75pt,y=0.75pt,yscale=-1,xscale=1]

\draw   (148,140.33) .. controls (148,131.86) and (154.86,125) .. (163.33,125) .. controls (171.8,125) and (178.67,131.86) .. (178.67,140.33) .. controls (178.67,148.8) and (171.8,155.67) .. (163.33,155.67) .. controls (154.86,155.67) and (148,148.8) .. (148,140.33) -- cycle ;
\draw   (226,140.33) .. controls (226,131.86) and (232.86,125) .. (241.33,125) .. controls (249.8,125) and (256.67,131.86) .. (256.67,140.33) .. controls (256.67,148.8) and (249.8,155.67) .. (241.33,155.67) .. controls (232.86,155.67) and (226,148.8) .. (226,140.33) -- cycle ;
\draw    (163.33,125) .. controls (183.67,103) and (227.67,103) .. (241.33,125) ;
\draw    (163.33,155.67) .. controls (185.67,174) and (218.67,174) .. (241.33,155.67) ;
\draw    (125.67,112) -- (151.67,131) ;
\draw    (148,140.33) -- (125.67,140) ;
\draw    (152.67,150.33) -- (128.67,168.33) ;
\draw   (384,141.33) .. controls (384,132.86) and (390.86,126) .. (399.33,126) .. controls (407.8,126) and (414.67,132.86) .. (414.67,141.33) .. controls (414.67,149.8) and (407.8,156.67) .. (399.33,156.67) .. controls (390.86,156.67) and (384,149.8) .. (384,141.33) -- cycle ;
\draw   (462,141.33) .. controls (462,132.86) and (468.86,126) .. (477.33,126) .. controls (485.8,126) and (492.67,132.86) .. (492.67,141.33) .. controls (492.67,149.8) and (485.8,156.67) .. (477.33,156.67) .. controls (468.86,156.67) and (462,149.8) .. (462,141.33) -- cycle ;
\draw    (361.67,113) -- (387.67,132) ;
\draw    (384,141.33) -- (361.67,141) ;
\draw    (388.67,151.33) -- (364.67,169.33) ;
\draw    (414.67,141.33) -- (462,141.33) ;
\draw    (492.67,141.33) -- (516.67,141.33) ;

\draw (148,135.4) node [anchor=north west][inner sep=0.75pt]  [font=\scriptsize]  {$g-2$};
\draw (237,136.4) node [anchor=north west][inner sep=0.75pt]  [font=\scriptsize]  {$1$};
\draw (106,155) node [anchor=north west][inner sep=0.75pt]    {$...$};
\draw (106,103.4) node [anchor=north west][inner sep=0.75pt]  [font=\scriptsize]  {$1$};
\draw (106,133.4) node [anchor=north west][inner sep=0.75pt]  [font=\scriptsize]  {$2$};
\draw (107,169.4) node [anchor=north west][inner sep=0.75pt]  [font=\scriptsize]  {$n$};
\draw (163,58.4) node [anchor=north west][inner sep=0.75pt] [font=\Large]    {$\Gamma_1$};
\draw (384,136.4) node [anchor=north west][inner sep=0.75pt]  [font=\scriptsize]  {$g-1$};
\draw (473,137.4) node [anchor=north west][inner sep=0.75pt]  [font=\scriptsize]  {$1$};
\draw (342,155.4) node [anchor=north west][inner sep=0.75pt]    {$...$};
\draw (342,104.4) node [anchor=north west][inner sep=0.75pt]  [font=\scriptsize]  {$1$};
\draw (342,134.4) node [anchor=north west][inner sep=0.75pt]  [font=\scriptsize]  {$2$};
\draw (343,170.4) node [anchor=north west][inner sep=0.75pt]  [font=\scriptsize]  {$n-1$};
\draw (399,59.4) node [anchor=north west][inner sep=0.75pt]  [font=\Large]  {$\Gamma_2$};
\draw (523,136.4) node [anchor=north west][inner sep=0.75pt]  [font=\scriptsize]  {$n$};

\end{tikzpicture}}

    \caption{Dual graphs of strata in the moduli space of curves containing a factor isomorphic to $\M_{1,2}$. The stratum corresponding to $\Gamma_1$ exists when $g\geq 2, n\geq 0$, whereas the stratum corresponding to $\Gamma_2$ exists when $g\geq 1, n\geq 3$, so together they cover the entire pseudo-stable range for $(g,n)$ except for $(1,2)$ which we consider as a base case and treat in Section \ref{sec:m12}.}
    \label{fig:dualgraphs}
\end{figure}

Let $\Gamma_1$, $\Gamma_2$ be the  graphs depicted in Figure \ref{fig:dualgraphs}. In the case $g\geq 2$ let $\Gamma = \Gamma_1$, corresponding to a genus $g-2$ vertex with $n+2$ legs and a genus $1$ vertex with $2$ legs, which are glued to two of the legs of the genus $g-2$ vertex.
Consider the corresponding diagram:
\begin{equation}
\label{eq:pullbackofdeltanotgeneralgn}
\xymatrix{
 \mathcal{\overline{M}}_{1,2}
 \ar[d]_{\mathcal{T}}
 \ar[rr]^{\iota_\xi}
 & &
 \mathcal{\overline{M}}_{g-2,n+2}\times \mathcal{\overline{M}}_{1,2}
 \ar[rr]^{gl_\Gamma}
 \ar[d]^{\mathcal{T}\times \mathcal{T}}& &\mathcal{\overline{M}}_{g,n} \ar[d]^{\mathcal{T}}\\
 \mathcal{M}_{1,2}^\ps
\ar[rr]^{\iota_\xi}
 & &
 \mathcal{M}_{g-2,n+2}^\ps\times \mathcal{M}_{1,2}^\ps
 \ar[rr]^{gl_\Gamma}& &\mathcal{M}_{g,n}^\ps,
}\end{equation}
where the inclusion maps $\iota_\xi$ have a constant value in the first factor corresponding to a chosen general curve $\xi\in \mathcal{M}_{g-2,n+2}$.

Imposing the equality of the two composite pull-backs in diagram \eqref{eq:pullbackofdeltanotgeneralgn} of the class $\delta_0^\ps$ to $\M_{1,2}$, one determines $q = 24$
in \eqref{eq:x} for $g\geq 2$. In the case $g=1, n\geq 3$, one may repeat this argument choosing $\Gamma = \Gamma_2$ to obtain that $q = 24$ in this case as well. 
Since the graph  corresponding to $\delta_0$ has a group of automorphisms of order two, we have that $\delta_0 = 2\alpha^\ast(\varphi_0)$ (and similarly in the pseudostable case).
Taking inverse images via $\alpha^\ast$ of all the terms, \eqref{eq:x} becomes:
\[
2\ \trop(\mathcal{T})^\ast(\varphi_0^\ps) = 2\varphi_0+24\varphi_1.
\]
The statement of the Theorem immediately follows.

\end{proof}

\subsection{The cuspidal locus}
In $\Mps_{g,n}$ there is a closed locus $\xi$ of points parameterizing curves with cuspidal singularities, corresponding to the (contracted) image of the divisor $\delta_1$ via the  morphism $\mathcal{T}$. 
We have 
$$\xi  \subset{\delta_0^\ps},$$ as any cuspidal curve may be obtained as the image of an element in $\delta_0 \cap \delta_1$, where the elliptic curve has become nodal. So far we have defined $\xi$ just as a geometric locus, but we impose a cycle structure by defining $\xi := \mathcal{T}_\ast(\delta_0 \cdot \delta_1)$.

\begin{proposition}\label{prop:cusp}
   The class of the cuspidal locus in $A^\ast(\Mps_{g,n})$  is
   \begin{equation}\label{eq:cusppp}
    \xi = \frac{(\delta_0^\ps)^2 - \mathcal{T}_\ast(\delta_0^2) }{24}.
\end{equation}
\end{proposition}
\begin{proof}

We use $\mathcal{T}^\ast(\delta^\ps_0) = \delta_0+24\delta_1$ (see Lemma \ref{lem:comp24} and Theorem \ref{thm:PLtropmap}), and the projection formula
\begin{equation}
{(\delta_0^\ps)}^2 = {\delta_0^\ps} \cdot \mathcal{T}_\ast(\delta_0)=\mathcal{T}_\ast(\mathcal{T}^\ast({\delta_0^\ps}) \cdot \delta_0) = \mathcal{T}_\ast((\delta_0+ 24 \delta_1) \cdot \delta_0) 
{=} \mathcal{T}_\ast(\delta_0^2) + 24\xi ,
\end{equation}
from which \eqref{eq:cusp} follows.
\end{proof}

We now reinterpret equation \eqref{eq:cusppp} in the language of piecewise polynomial functions from Section~\ref{pwpcalculus}.  Recall that by $\Phi_0$ we denote the piecewise quadratic function that restricts to 
$x_0^2$ on the ray $\rho_0$, to zero on all other rays, and to the function
$\sum_{i\in I_0} x_i^2$  on all cones which has rays identified with $\rho_0$  (possible multiple ones), where $I_0$ denotes the set of linear coordinates dual to rays that get identified with $\rho_0$.  Given a $k$-dimensional cone $\sigma$ in $M_{g,n}^{\trop}$, $\varphi_\sigma$ denotes the degree $k$ square free monic monomial function with support on the cone $\sigma$. We add a superscript $\ps$ to denote an analogous object in the pseudostable world. 

For every cone $\sigma$ that does not contain $\rho_1$ as a ray, we clearly have that 
\begin{equation} \label{eq:pfstrata}
    \mathcal{T}_\ast(\alpha^\ast(\varphi_\sigma)) = \alpha_\ps^\ast({\varphi}^\ps_\sigma).
\end{equation}

Denote by $B$  the set of two dimensional cones in $M_{g,n}^{\trop}$ whose only ray is $\rho_0$. These cones index the strata commonly called ``banana" curves, which are contained in the transversal part of  $\delta_0^2$. We have
\begin{equation}\label{eq:phisquare}
 \varphi_0^2 = \Phi_0 + 2\sum_{\sigma\in B} \varphi_\sigma,  
\end{equation}
and the analogous expression holds in $\Mps_{g,n}$. We have $\delta_0 = \alpha^\ast(\varphi_0)$, so we can express the class of the cuspidal locus in terms of piecewise polynomial functions.
After plugging in \eqref{eq:phisquare} in   \eqref{eq:cusppp}, we see that by \eqref{eq:pfstrata} all the square free terms  cancel, yielding
\begin{equation}
\label{eq:cuspp}
    \xi = \frac{\alpha^\ast(\Phi_0^\ps) - \mathcal{T}_\ast(\alpha^\ast(\Phi_0)) }{12}.
\end{equation}

\section{Genus 1} \label{sec:genus1}

\subsection{The base case}\label{sec:m12}
In the setting of genus 1 with two marked points, the contraction $$\mathcal{T}:\M_{1,2}\longrightarrow \mathcal{M}^{\rm ps}_{1,2}$$ and its tropicalization can be easily visualized. Both moduli spaces are two dimensional. The source $\M_{1,2}$ has two one-dimensional boundary strata: the divisor $\delta_0$ of irreducible singular curves and $\delta_{1}$, the elliptic tail locus. In $\mathcal{M}^{\rm ps}_{1,2}$ the elliptic tail locus has contracted, and its contracted image is the cuspidal locus, consisting of a single point,  which is not a logarithmic stratum (see Figure~\ref{fig:cont12}).
 \begin{lemma}
     \label{lem:comp24}
     When $(g,n)= (1,2)$, we have
     \[
     \mathcal{T}^\ast(\delta_0^\ps) = \delta_0 +24 \delta_1.
     \]
 \end{lemma}
 \begin{proof}
     Observe that $\mathcal{T}^\ast(\delta_0^\ps)$ must have the form $ \delta_0 +q \delta_1$ for some rational number $q$. Since the constant map to a point factors through $\mathcal{T}$, one may deduce from the projection formula (and specifically the fact that $\mathcal{T}_\ast(\delta_1) = 0$) that
     \[
     \int_{\M_{1,2}}  \mathcal{T}^\ast(\delta_0^\ps)\cdot \delta_1 = \int_{\M_{1,2}}  (\delta_0 +q \delta_1)\cdot \delta_1 = 0.
     \]
     On $\M_{1,2}$ we have $\int \delta_0\cdot \delta_1 = 1$ and $\int  \delta_1^2 = -1/24$, from which the statement of the Lemma follows.
 \end{proof}

The generalized complex for $\mathcal{M}^{\rm trop}_{1,2}$ and the combinatorial graph types corresponding to each of the cones are depicted in Figure~\ref{fig:genus1ps}. This complex has a ray in the positive $x$ coordinate corresponding to the boundary divisor $\delta_0$, and one in the $y$-coordinate corresponding to $\delta_{1}$. The complex for $\mathcal{M}_{1,2}^{\trop, \rm ps}$ consists only of the single $\delta_0$ ray and the folded two dimensional cone glued to it.

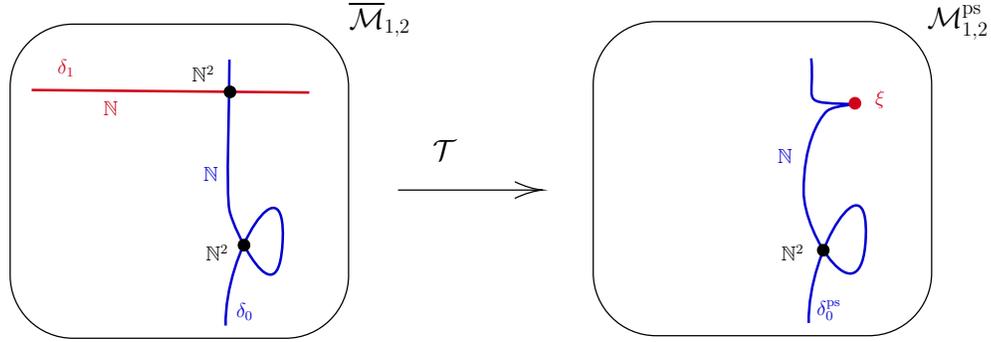
\begin{figure}[tb]
    \centering

\resizebox{.8\textwidth}{!}{
\tikzset{every picture/.style={line width=0.75pt}} 

\begin{tikzpicture}[x=0.75pt,y=0.75pt,yscale=-1,xscale=1]

\draw  [color={rgb, 255:red, 0; green, 0; blue, 0 }  ,draw opacity=1 ] (19,85.6) .. controls (19,58.21) and (41.21,36) .. (68.6,36) -- (236.4,36) .. controls (263.79,36) and (286,58.21) .. (286,85.6) -- (286,234.4) .. controls (286,261.79) and (263.79,284) .. (236.4,284) -- (68.6,284) .. controls (41.21,284) and (19,261.79) .. (19,234.4) -- cycle ;
\draw [color={rgb, 255:red, 2; green, 2; blue, 208 }  ,draw opacity=1 ][line width=1.5]    (192,64) .. controls (192,100) and (189.08,167.79) .. (192,183) .. controls (194.92,198.21) and (231,273) .. (234,207) .. controls (237,141) and (188,215) .. (189,274) ;
\draw  [color={rgb, 255:red, 0; green, 0; blue, 0 }  ,draw opacity=1 ] (479,83.6) .. controls (479,56.21) and (501.21,34) .. (528.6,34) -- (696.4,34) .. controls (723.79,34) and (746,56.21) .. (746,83.6) -- (746,232.4) .. controls (746,259.79) and (723.79,282) .. (696.4,282) -- (528.6,282) .. controls (501.21,282) and (479,259.79) .. (479,232.4) -- cycle ;
\draw [color={rgb, 255:red, 2; green, 2; blue, 208 }  ,draw opacity=1 ][line width=1.5]    (651,63) .. controls (654,96) and (642.15,97.58) .. (673.08,98.79) .. controls (704,100) and (670,98) .. (663,105) .. controls (656,112) and (645,132) .. (645,168) .. controls (645,204) and (691,271) .. (694,205) .. controls (697,139) and (648,213) .. (649,272) ;
\draw  [color={rgb, 255:red, 208; green, 2; blue, 27 }  ,draw opacity=1 ][fill={rgb, 255:red, 208; green, 2; blue, 27 }  ,fill opacity=1 ] (681,98.5) .. controls (681,96.01) and (683.01,94) .. (685.5,94) .. controls (687.99,94) and (690,96.01) .. (690,98.5) .. controls (690,100.99) and (687.99,103) .. (685.5,103) .. controls (683.01,103) and (681,100.99) .. (681,98.5) -- cycle ;
\draw [color={rgb, 255:red, 208; green, 2; blue, 27 }  ,draw opacity=1 ][line width=1.5]    (36,88) -- (255,90) ;
\draw    (325,167) -- (437,167) ;
\draw [shift={(439,167)}, rotate = 180] [color={rgb, 255:red, 0; green, 0; blue, 0 }  ][line width=0.75]    (21.86,-6.58) .. controls (13.9,-2.79) and (6.61,-0.6) .. (0,0) .. controls (6.61,0.6) and (13.9,2.79) .. (21.86,6.58)   ;
\draw  [color={rgb, 255:red, 0; green, 0; blue, 0 }  ,draw opacity=1 ][fill={rgb, 255:red, 0; green, 0; blue, 0 }  ,fill opacity=1 ] (656,214.5) .. controls (656,212.01) and (658.01,210) .. (660.5,210) .. controls (662.99,210) and (665,212.01) .. (665,214.5) .. controls (665,216.99) and (662.99,219) .. (660.5,219) .. controls (658.01,219) and (656,216.99) .. (656,214.5) -- cycle ;
\draw  [color={rgb, 255:red, 0; green, 0; blue, 0 }  ,draw opacity=1 ][fill={rgb, 255:red, 0; green, 0; blue, 0 }  ,fill opacity=1 ] (199,210.5) .. controls (199,208.01) and (201.01,206) .. (203.5,206) .. controls (205.99,206) and (208,208.01) .. (208,210.5) .. controls (208,212.99) and (205.99,215) .. (203.5,215) .. controls (201.01,215) and (199,212.99) .. (199,210.5) -- cycle ;
\draw  [color={rgb, 255:red, 0; green, 0; blue, 0 }  ,draw opacity=1 ][fill={rgb, 255:red, 0; green, 0; blue, 0 }  ,fill opacity=1 ] (188,89.5) .. controls (188,87.01) and (190.01,85) .. (192.5,85) .. controls (194.99,85) and (197,87.01) .. (197,89.5) .. controls (197,91.99) and (194.99,94) .. (192.5,94) .. controls (190.01,94) and (188,91.99) .. (188,89.5) -- cycle ;

\draw (55,62.4) node [anchor=north west][inner sep=0.75pt]    {$\textcolor[rgb]{0.82,0.01,0.11}{\delta_1}$};
\draw (196,254.4) node [anchor=north west][inner sep=0.75pt]    {$\textcolor[rgb]{0.01,0.01,0.82}{\delta_0}$};

\draw (700,87) node [anchor=north west][inner sep=0.75pt]    {$\textcolor[rgb]{0.82,0.01,0.11}{\xi}$};
\draw (654,252.4) node [anchor=north west][inner sep=0.75pt]    {$\textcolor[rgb]{0.01,0.01,0.82}{\delta_0^\ps}$};
\draw (351,126.4) node [anchor=north west][inner sep=0.75pt]  [font=\Large]  {$\mathcal{T}$};
\draw (285,17.4) node [anchor=north west][inner sep=0.75pt]  [font=\Large]  {$\M_{1,2}$};
\draw (741,19.4) node [anchor=north west][inner sep=0.75pt]  [font=\Large]  {$\mathcal{M}_{1,2}^\ps$};
\draw (170,147.4) node [anchor=north west][inner sep=0.75pt]  [color={rgb, 255:red, 2; green, 2; blue, 208 }  ,opacity=1 ]  {$\mathbb{N}$};
\draw (91,95.4) node [anchor=north west][inner sep=0.75pt]  [color={rgb, 255:red, 208; green, 2; blue, 27 }  ,opacity=1 ]  {$\mathbb{N}$};
\draw (623,132.9) node [anchor=north west][inner sep=0.75pt]  [color={rgb, 255:red, 2; green, 2; blue, 208 }  ,opacity=1 ]  {$\mathbb{N}$};
\draw (172,208.4) node [anchor=north west][inner sep=0.75pt]    {$\mathbb{N}^2$};
\draw (161,67.4) node [anchor=north west][inner sep=0.75pt]    {$\mathbb{N}^2$};
\draw (626,208.4) node [anchor=north west][inner sep=0.75pt]    {$\mathbb{N}^2$};

\end{tikzpicture}}

    \caption{The contraction of $\delta_1$ along the morphism $\mathcal{T}:\M_{1,2}\longrightarrow \Mps_{1,2}$. We observe that the cuspidal locus $\xi$ is not a logarithmic stratum, and it is a singular point of the divisor $\delta_0^\ps$.}
    \label{fig:cont12}
\end{figure}

Theorem \ref{thm:PLtropmap} asserts that the map $\trop(\mathcal{T})$ \emph{does not} contract the $\delta_{1}$ ray in the $y$-coordinate like a projection, but instead contracts the entire two dimensional cone onto the $\delta_0$ ray linearly by sending $(x,y)$ to $(0, 12x+y)$. 

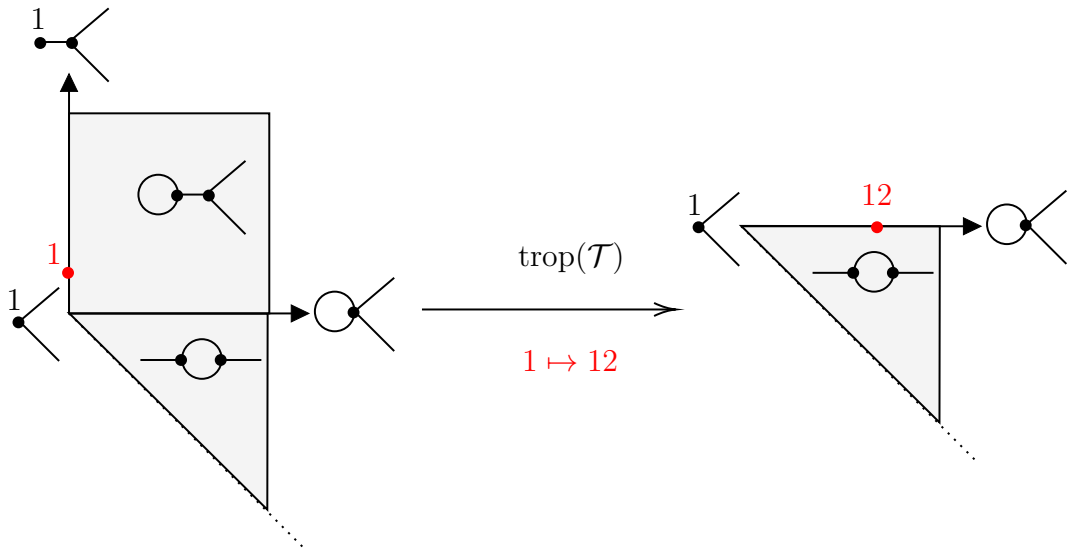
\begin{figure} 
\tikzset{
pattern size/.store in=\mcSize, 
pattern size = 5pt,
pattern thickness/.store in=\mcThickness, 
pattern thickness = 0.3pt,
pattern radius/.store in=\mcRadius, 
pattern radius = 1pt}
\makeatletter
\pgfutil@ifundefined{pgf@pattern@name@_f86ig3vaq}{
\makeatletter
\pgfdeclarepatternformonly[\mcRadius,\mcThickness,\mcSize]{_f86ig3vaq}
{\pgfpoint{-0.5*\mcSize}{-0.5*\mcSize}}
{\pgfpoint{0.5*\mcSize}{0.5*\mcSize}}
{\pgfpoint{\mcSize}{\mcSize}}
{
\pgfsetcolor{\tikz@pattern@color}
\pgfsetlinewidth{\mcThickness}
\pgfpathcircle\pgfpointorigin{\mcRadius}
\pgfusepath{stroke}
}}
\makeatother

 
\tikzset{
pattern size/.store in=\mcSize, 
pattern size = 5pt,
pattern thickness/.store in=\mcThickness, 
pattern thickness = 0.3pt,
pattern radius/.store in=\mcRadius, 
pattern radius = 1pt}
\makeatletter
\pgfutil@ifundefined{pgf@pattern@name@_hreesoesb}{
\makeatletter
\pgfdeclarepatternformonly[\mcRadius,\mcThickness,\mcSize]{_hreesoesb}
{\pgfpoint{-0.5*\mcSize}{-0.5*\mcSize}}
{\pgfpoint{0.5*\mcSize}{0.5*\mcSize}}
{\pgfpoint{\mcSize}{\mcSize}}
{
\pgfsetcolor{\tikz@pattern@color}
\pgfsetlinewidth{\mcThickness}
\pgfpathcircle\pgfpointorigin{\mcRadius}
\pgfusepath{stroke}
}}
\makeatother
\tikzset{every picture/.style={line width=0.75pt}} 

\begin{tikzpicture}[x=0.75pt,y=0.75pt,yscale=-1,xscale=1]

\draw  [draw opacity=0][fill={rgb, 255:red, 244; green, 244; blue, 244 }  ,fill opacity=1 ] (100,220.82) -- (200,320) -- (200.17,221) -- cycle ;
\draw  [draw opacity=0][fill={rgb, 255:red, 244; green, 244; blue, 244 }  ,fill opacity=1 ] (100,120) -- (201,120) -- (201,221) -- (100,221) -- cycle ;
\draw    (100,103) -- (100,221) ;
\draw [shift={(100,100)}, rotate = 90] [fill={rgb, 255:red, 0; green, 0; blue, 0 }  ][line width=0.08]  [draw opacity=0] (8.93,-4.29) -- (0,0) -- (8.93,4.29) -- cycle    ;
\draw    (100,221) -- (214,221) -- (218,221) ;
\draw [shift={(221,221)}, rotate = 180] [fill={rgb, 255:red, 0; green, 0; blue, 0 }  ][line width=0.08]  [draw opacity=0] (8.93,-4.29) -- (0,0) -- (8.93,4.29) -- cycle    ;
\draw [pattern=_f86ig3vaq,pattern size=6pt,pattern thickness=0.75pt,pattern radius=0.75pt, pattern color={rgb, 255:red, 0; green, 0; blue, 0}][line width=0.75]  [dash pattern={on 0.84pt off 2.51pt}]  (100,221) -- (220,341) ;
\draw   (135,161) .. controls (135,155.48) and (139.48,151) .. (145,151) .. controls (150.52,151) and (155,155.48) .. (155,161) .. controls (155,166.52) and (150.52,171) .. (145,171) .. controls (139.48,171) and (135,166.52) .. (135,161) -- cycle ;
\draw    (155,161) -- (169,161) ;
\draw    (169,161) -- (189,144) ;
\draw    (169,161) -- (189,181) ;
\draw  [fill={rgb, 255:red, 0; green, 0; blue, 0 }  ,fill opacity=1 ] (152,161.5) .. controls (152,160.12) and (153.12,159) .. (154.5,159) .. controls (155.88,159) and (157,160.12) .. (157,161.5) .. controls (157,162.88) and (155.88,164) .. (154.5,164) .. controls (153.12,164) and (152,162.88) .. (152,161.5) -- cycle ;
\draw  [fill={rgb, 255:red, 0; green, 0; blue, 0 }  ,fill opacity=1 ] (168,161.5) .. controls (168,160.12) and (169.12,159) .. (170.5,159) .. controls (171.88,159) and (173,160.12) .. (173,161.5) .. controls (173,162.88) and (171.88,164) .. (170.5,164) .. controls (169.12,164) and (168,162.88) .. (168,161.5) -- cycle ;
\draw   (224,220) .. controls (224,214.48) and (228.48,210) .. (234,210) .. controls (239.52,210) and (244,214.48) .. (244,220) .. controls (244,225.52) and (239.52,230) .. (234,230) .. controls (228.48,230) and (224,225.52) .. (224,220) -- cycle ;
\draw    (244,220) -- (264,203) ;
\draw    (244,220) -- (264,240) ;
\draw  [fill={rgb, 255:red, 0; green, 0; blue, 0 }  ,fill opacity=1 ] (241,220.5) .. controls (241,219.12) and (242.12,218) .. (243.5,218) .. controls (244.88,218) and (246,219.12) .. (246,220.5) .. controls (246,221.88) and (244.88,223) .. (243.5,223) .. controls (242.12,223) and (241,221.88) .. (241,220.5) -- cycle ;
\draw   (157,244) .. controls (157,238.48) and (161.48,234) .. (167,234) .. controls (172.52,234) and (177,238.48) .. (177,244) .. controls (177,249.52) and (172.52,254) .. (167,254) .. controls (161.48,254) and (157,249.52) .. (157,244) -- cycle ;
\draw    (179,244.5) -- (197,244.5) ;
\draw  [fill={rgb, 255:red, 0; green, 0; blue, 0 }  ,fill opacity=1 ] (174,244.5) .. controls (174,243.12) and (175.12,242) .. (176.5,242) .. controls (177.88,242) and (179,243.12) .. (179,244.5) .. controls (179,245.88) and (177.88,247) .. (176.5,247) .. controls (175.12,247) and (174,245.88) .. (174,244.5) -- cycle ;
\draw    (136,244.5) -- (154,244.5) ;
\draw  [fill={rgb, 255:red, 0; green, 0; blue, 0 }  ,fill opacity=1 ] (154,244.5) .. controls (154,243.12) and (155.12,242) .. (156.5,242) .. controls (157.88,242) and (159,243.12) .. (159,244.5) .. controls (159,245.88) and (157.88,247) .. (156.5,247) .. controls (155.12,247) and (154,245.88) .. (154,244.5) -- cycle ;
\draw    (75,225) -- (95,208) ;
\draw    (75,225) -- (95,245) ;
\draw  [fill={rgb, 255:red, 0; green, 0; blue, 0 }  ,fill opacity=1 ] (72,225.5) .. controls (72,224.12) and (73.12,223) .. (74.5,223) .. controls (75.88,223) and (77,224.12) .. (77,225.5) .. controls (77,226.88) and (75.88,228) .. (74.5,228) .. controls (73.12,228) and (72,226.88) .. (72,225.5) -- cycle ;
\draw    (86,84) -- (100,84) ;
\draw    (100,84) -- (120,67) ;
\draw    (100,84) -- (120,104) ;
\draw  [fill={rgb, 255:red, 0; green, 0; blue, 0 }  ,fill opacity=1 ] (83,84.5) .. controls (83,83.12) and (84.12,82) .. (85.5,82) .. controls (86.88,82) and (88,83.12) .. (88,84.5) .. controls (88,85.88) and (86.88,87) .. (85.5,87) .. controls (84.12,87) and (83,85.88) .. (83,84.5) -- cycle ;
\draw  [fill={rgb, 255:red, 0; green, 0; blue, 0 }  ,fill opacity=1 ] (99,84.5) .. controls (99,83.12) and (100.12,82) .. (101.5,82) .. controls (102.88,82) and (104,83.12) .. (104,84.5) .. controls (104,85.88) and (102.88,87) .. (101.5,87) .. controls (100.12,87) and (99,85.88) .. (99,84.5) -- cycle ;
\draw  [color={rgb, 255:red, 255; green, 0; blue, 0 }  ,draw opacity=1 ][fill={rgb, 255:red, 255; green, 0; blue, 0 }  ,fill opacity=1 ] (97,200.5) .. controls (97,199.12) and (98.12,198) .. (99.5,198) .. controls (100.88,198) and (102,199.12) .. (102,200.5) .. controls (102,201.88) and (100.88,203) .. (99.5,203) .. controls (98.12,203) and (97,201.88) .. (97,200.5) -- cycle ;
\draw    (278,219) -- (404,219) ;
\draw [shift={(406,219)}, rotate = 180] [color={rgb, 255:red, 0; green, 0; blue, 0 }  ][line width=0.75]    (10.93,-3.29) .. controls (6.95,-1.4) and (3.31,-0.3) .. (0,0) .. controls (3.31,0.3) and (6.95,1.4) .. (10.93,3.29)   ;
\draw  [draw opacity=0][fill={rgb, 255:red, 244; green, 244; blue, 244 }  ,fill opacity=1 ] (439,176.82) -- (539,276) -- (539.17,177) -- cycle ;
\draw    (439,177) -- (553,177) -- (557,177) ;
\draw [shift={(560,177)}, rotate = 180] [fill={rgb, 255:red, 0; green, 0; blue, 0 }  ][line width=0.08]  [draw opacity=0] (8.93,-4.29) -- (0,0) -- (8.93,4.29) -- cycle    ;
\draw [pattern=_hreesoesb,pattern size=6pt,pattern thickness=0.75pt,pattern radius=0.75pt, pattern color={rgb, 255:red, 0; green, 0; blue, 0}][line width=0.75]  [dash pattern={on 0.84pt off 2.51pt}]  (439,177) -- (559,297) ;
\draw   (563,176) .. controls (563,170.48) and (567.48,166) .. (573,166) .. controls (578.52,166) and (583,170.48) .. (583,176) .. controls (583,181.52) and (578.52,186) .. (573,186) .. controls (567.48,186) and (563,181.52) .. (563,176) -- cycle ;
\draw    (583,176) -- (603,159) ;
\draw    (583,176) -- (603,196) ;
\draw  [fill={rgb, 255:red, 0; green, 0; blue, 0 }  ,fill opacity=1 ] (580,176.5) .. controls (580,175.12) and (581.12,174) .. (582.5,174) .. controls (583.88,174) and (585,175.12) .. (585,176.5) .. controls (585,177.88) and (583.88,179) .. (582.5,179) .. controls (581.12,179) and (580,177.88) .. (580,176.5) -- cycle ;
\draw   (496,200) .. controls (496,194.48) and (500.48,190) .. (506,190) .. controls (511.52,190) and (516,194.48) .. (516,200) .. controls (516,205.52) and (511.52,210) .. (506,210) .. controls (500.48,210) and (496,205.52) .. (496,200) -- cycle ;
\draw    (518,200.5) -- (536,200.5) ;
\draw  [fill={rgb, 255:red, 0; green, 0; blue, 0 }  ,fill opacity=1 ] (513,200.5) .. controls (513,199.12) and (514.12,198) .. (515.5,198) .. controls (516.88,198) and (518,199.12) .. (518,200.5) .. controls (518,201.88) and (516.88,203) .. (515.5,203) .. controls (514.12,203) and (513,201.88) .. (513,200.5) -- cycle ;
\draw    (475,200.5) -- (493,200.5) ;
\draw  [fill={rgb, 255:red, 0; green, 0; blue, 0 }  ,fill opacity=1 ] (493,200.5) .. controls (493,199.12) and (494.12,198) .. (495.5,198) .. controls (496.88,198) and (498,199.12) .. (498,200.5) .. controls (498,201.88) and (496.88,203) .. (495.5,203) .. controls (494.12,203) and (493,201.88) .. (493,200.5) -- cycle ;
\draw    (418,177) -- (438,160) ;
\draw    (418,177) -- (438,197) ;
\draw  [fill={rgb, 255:red, 0; green, 0; blue, 0 }  ,fill opacity=1 ] (415,177.5) .. controls (415,176.12) and (416.12,175) .. (417.5,175) .. controls (418.88,175) and (420,176.12) .. (420,177.5) .. controls (420,178.88) and (418.88,180) .. (417.5,180) .. controls (416.12,180) and (415,178.88) .. (415,177.5) -- cycle ;
\draw  [color={rgb, 255:red, 255; green, 0; blue, 0 }  ,draw opacity=1 ][fill={rgb, 255:red, 255; green, 0; blue, 0 }  ,fill opacity=1 ] (505,177.5) .. controls (505,176.12) and (506.12,175) .. (507.5,175) .. controls (508.88,175) and (510,176.12) .. (510,177.5) .. controls (510,178.88) and (508.88,180) .. (507.5,180) .. controls (506.12,180) and (505,178.88) .. (505,177.5) -- cycle ;

\draw (67,207.4) node [anchor=north west][inner sep=0.75pt]    {$1$};
\draw (79,65.4) node [anchor=north west][inner sep=0.75pt]    {$1$};
\draw (87,184.4) node [anchor=north west][inner sep=0.75pt]  [color={rgb, 255:red, 255; green, 0; blue, 0 }  ,opacity=1 ]  {$1$};
\draw (410,159.4) node [anchor=north west][inner sep=0.75pt]    {$1$};
\draw (498,154.4) node [anchor=north west][inner sep=0.75pt]  [color={rgb, 255:red, 255; green, 0; blue, 0 }  ,opacity=1 ]  {$12$};
\draw (325,183.4) node [anchor=north west][inner sep=0.75pt]    {$\trop(\mathcal{T})$};
\draw (327,238.4) node [anchor=north west][inner sep=0.75pt]    {${\color{red} 1 \mapsto 12}$};
\end{tikzpicture}
\caption{The morphism of generalized cone complexes $\trop(\mathcal{T})$ from $\mathcal{M}^{\rm trop}_{1,2}$ to $\mathcal{M}_{1,2}^{\trop, \rm ps}.$ }\label{fig:genus1ps}
\end{figure}

\subsection{The cuspidal locus}
We begin by observing that for $(g,n) = (1,2)$ the cuspidal locus is a single point. Its tangent space is two-dimensional and it may be identified with the mini-versal deformation space of the cusp. It is well known \cite[page 98]{harrismorrison} that the latter agrees with the space of Weierstrass equations for elliptic curves, i.e. the point with coordinates $(p,q)$ corresponds to the curve: 
\[
y^2 = x^3+px+q.
\]
The locus of nodal curves of arithmetic genus one is parameterized by the discriminant locus $4p^3+27q^2 = 0$, which is itself a cuspidal cubic in the $(p,q)$ plane. Hence the irreducible pseudostable divisor $\delta_0^{\ps}$ acquires a cuspidal singularity precisely at the cuspidal point.

This analysis  allows for a geometric interpretation for the factor of $24$ computed in Lemma \ref{lem:comp24}. A normal crossing resolution of a plane cuspidal singularity requires a chain of three blow-ups (see e.g. \cite[Example 3.9.1]{hartshorne}), and the local picture of the boundary at the point of intersection of the proper transform of the cuspidal curve with the exceptional locus consists of a nodal curve, with one branch being the proper transform and the other an exceptional rational curve, which appears with multiplicity $6$ in the total transform of the cuspidal curve. This exceptional curve is in fact the coarse space of $\delta_1$. There are two additional factors of $2$ that give us the final factor of $24$ in Lemma \ref{lem:comp24}: the first is due to the fact that $\delta_1$ is a $\mu_2$-gerbe over its coarse space; the second comes from the convention of denoting by $\delta_0^{\ps}$ the pushforward of the fundamental class via the gluing morphism $gl: \Mps_{0,4}\to \Mps_{1,2}$, which is a degree two covering onto its image, hence when pulling-back $\delta_0^{ps}$ we are actually pulling back the class of twice the discriminant locus in the $(p,q)$ plane. 

We conclude this section with a corollary of Proposition \ref{prop:cusp}, showing that in genus one the class of the cuspidal locus is equivalent to a piecewise polynomial function on $\mathcal{M}_{1,n}^{\trop, \ps}$

\begin{corollary} \label{cor:g1poly}
When $g = 1, n\geq 2$, the class $\xi$ of the cuspidal locus is given by
\[
\xi = \alpha^\ast\left(\frac{(\varphi^\ps_{0})^2}{6}\right),
\]
where $\varphi_0$ denotes the piecewise linear function with slope $1$ along the ray  of $\mathcal{M}_{1,n}^{\trop, \ps}$ dual to $\delta_0$, and zero along any other ray.
\end{corollary}
\begin{proof}
    This follows from  \eqref{eq:cusp}: in genus one $\delta_0^2 = 0$, hence the equation reduces to $\xi = \frac{(\delta_0^\ps)^2}{24}$. Recalling that $\delta_0^\ps =2 \alpha^\ast(\varphi_0^\ps)$, we immediately obtain $\xi = \alpha^\ast\left(\frac{(\varphi^\ps_{0})^2}{6}\right)$.
\end{proof}

\subsection{The class $\lambda_1$} In this section we compare the class $\lambda_1$, i.e. the first Chern class of the Hodge bundle, on the spaces of stable and pseudostable curves of genus one. We use the description of the piecewise linear function $\mathcal{T}^\trop$, and give an alternative proof of the result of \cite[Theorem 2.4]{CGRVW}.

\begin{theorem}[\cite{CGRVW}]\label{thm:lala}
    For $g=1$ and  any $n> 1$, we have 
    \[
\mathcal{T}^\ast(\lambda_1^\ps) = \lambda_1 +\delta_1.    
    \]
\end{theorem}
\begin{proof}
We observe that for dimension  reasons (the two Hodge bundles may be identified away from the locus that gets contracted to codimension two) it must be that $\mathcal{T}_\ast(\lambda_1) = \lambda_1^\ps$. In genus $1$, the class $\lambda_1 = \delta_0/24$: this may be computed easily in $\overline{\mathcal{M}_{1,1}}$, and then extended to arbitrary $n$ by observing that both $\lambda_1$ and $\delta_0$ are stable under pull-back.
By pushing forward this boundary expression via $\mathcal{T}$, we obtain $\lambda_1^\ps = \delta_0^\ps/24$, which is in the image of the map $\alpha^\ast$ and can therefore be written as the piecewise polynomial $\varphi^\ps_0/12$ on $\mathcal{M}_{1,n}^{\trop, \rm ps}$.
It now follows from Theorem \ref{thm:PLtropmap} that
\begin{equation}
    \mathcal{T}^\ast(\lambda^\ps_1) = \alpha^\ast\left(\trop\mathcal{T}^{\  \ast}\left(\frac{\varphi^\ps_0}{12}\right)\right)= \alpha^\ast\left(\frac{\varphi_0}{12} + \varphi_1 \right) = \frac{\delta_0}{24}+\delta_1 = \lambda_1+\delta_1.
\end{equation}

\end{proof}

\subsection{Hassett light point spaces} We compare the contraction morphism $\mathcal{T}$ to a similar example where the tropical moduli spaces are identical but the piecewise linear morphism among them is different. For a fixed weight vector $\omega = (\omega_1,\ldots,\omega_n)\in (\mathbb{Q}\cap(0,1])^n$ with 
\[2g-2+\sum_{i=1}^{n} \omega_i >0, \]
Hassett's moduli space of weighted stable curves $\M_{g,\omega}$ parameterizes $n$-marked, $\omega$-weighted stable curves of genus $g$ (see \cite{hassettw}). The weight $\omega_i$ is assigned to the $i$-th marked point. A subset of the marked points may coincide if the sum of their weights does not exceed 1, and stability requires the weighted log canonical divisor $\omega_C+\sum_i\omega_ip_i$ to be ample. 

For a given weight vector $\omega$, the tropicalization  $\M_{g,\omega}^{\rm trop}$ parametrizes tropical curves with the added data of the weights $\omega_i$ on each marking. Stability in this setting requires the $\omega$-weighted valency of each genus 0 vertex $v$, that is, the number of half-edges belonging to compact edges incident to $v$ plus the sum of the weights of the markings at $v$, to be larger than 2. This tropical moduli space has been well-studied -- we refer the reader to \cite{CHM, uliw} for details.

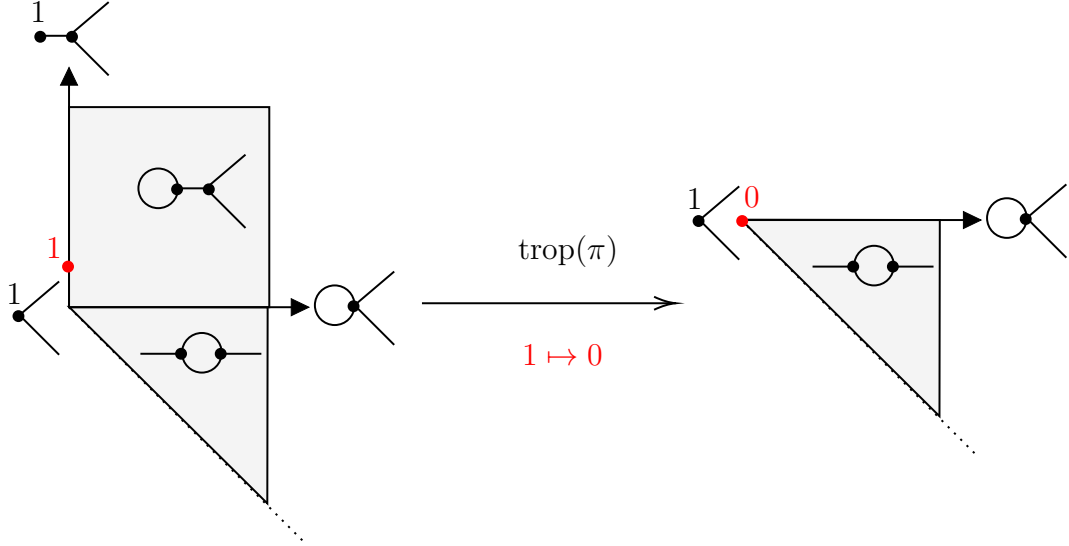
\begin{figure}[tb]

 
\tikzset{
pattern size/.store in=\mcSize, 
pattern size = 5pt,
pattern thickness/.store in=\mcThickness, 
pattern thickness = 0.3pt,
pattern radius/.store in=\mcRadius, 
pattern radius = 1pt}
\makeatletter
\pgfutil@ifundefined{pgf@pattern@name@_ajjd2icqg}{
\makeatletter
\pgfdeclarepatternformonly[\mcRadius,\mcThickness,\mcSize]{_ajjd2icqg}
{\pgfpoint{-0.5*\mcSize}{-0.5*\mcSize}}
{\pgfpoint{0.5*\mcSize}{0.5*\mcSize}}
{\pgfpoint{\mcSize}{\mcSize}}
{
\pgfsetcolor{\tikz@pattern@color}
\pgfsetlinewidth{\mcThickness}
\pgfpathcircle\pgfpointorigin{\mcRadius}
\pgfusepath{stroke}
}}
\makeatother

 
\tikzset{
pattern size/.store in=\mcSize, 
pattern size = 5pt,
pattern thickness/.store in=\mcThickness, 
pattern thickness = 0.3pt,
pattern radius/.store in=\mcRadius, 
pattern radius = 1pt}
\makeatletter
\pgfutil@ifundefined{pgf@pattern@name@_zciasiwn5}{
\makeatletter
\pgfdeclarepatternformonly[\mcRadius,\mcThickness,\mcSize]{_zciasiwn5}
{\pgfpoint{-0.5*\mcSize}{-0.5*\mcSize}}
{\pgfpoint{0.5*\mcSize}{0.5*\mcSize}}
{\pgfpoint{\mcSize}{\mcSize}}
{
\pgfsetcolor{\tikz@pattern@color}
\pgfsetlinewidth{\mcThickness}
\pgfpathcircle\pgfpointorigin{\mcRadius}
\pgfusepath{stroke}
}}
\makeatother
\tikzset{every picture/.style={line width=0.75pt}} 

\begin{tikzpicture}[x=0.75pt,y=0.75pt,yscale=-1,xscale=1]

\draw  [draw opacity=0][fill={rgb, 255:red, 244; green, 244; blue, 244 }  ,fill opacity=1 ] (100,220.82) -- (200,320) -- (200.17,221) -- cycle ;
\draw  [draw opacity=0][fill={rgb, 255:red, 244; green, 244; blue, 244 }  ,fill opacity=1 ] (100,120) -- (201,120) -- (201,221) -- (100,221) -- cycle ;
\draw    (100,103) -- (100,221) ;
\draw [shift={(100,100)}, rotate = 90] [fill={rgb, 255:red, 0; green, 0; blue, 0 }  ][line width=0.08]  [draw opacity=0] (8.93,-4.29) -- (0,0) -- (8.93,4.29) -- cycle    ;
\draw    (100,221) -- (214,221) -- (218,221) ;
\draw [shift={(221,221)}, rotate = 180] [fill={rgb, 255:red, 0; green, 0; blue, 0 }  ][line width=0.08]  [draw opacity=0] (8.93,-4.29) -- (0,0) -- (8.93,4.29) -- cycle    ;
\draw [pattern=_ajjd2icqg,pattern size=6pt,pattern thickness=0.75pt,pattern radius=0.75pt, pattern color={rgb, 255:red, 0; green, 0; blue, 0}][line width=0.75]  [dash pattern={on 0.84pt off 2.51pt}]  (100,221) -- (220,341) ;
\draw   (135,161) .. controls (135,155.48) and (139.48,151) .. (145,151) .. controls (150.52,151) and (155,155.48) .. (155,161) .. controls (155,166.52) and (150.52,171) .. (145,171) .. controls (139.48,171) and (135,166.52) .. (135,161) -- cycle ;
\draw    (155,161) -- (169,161) ;
\draw    (169,161) -- (189,144) ;
\draw    (169,161) -- (189,181) ;
\draw  [fill={rgb, 255:red, 0; green, 0; blue, 0 }  ,fill opacity=1 ] (152,161.5) .. controls (152,160.12) and (153.12,159) .. (154.5,159) .. controls (155.88,159) and (157,160.12) .. (157,161.5) .. controls (157,162.88) and (155.88,164) .. (154.5,164) .. controls (153.12,164) and (152,162.88) .. (152,161.5) -- cycle ;
\draw  [fill={rgb, 255:red, 0; green, 0; blue, 0 }  ,fill opacity=1 ] (168,161.5) .. controls (168,160.12) and (169.12,159) .. (170.5,159) .. controls (171.88,159) and (173,160.12) .. (173,161.5) .. controls (173,162.88) and (171.88,164) .. (170.5,164) .. controls (169.12,164) and (168,162.88) .. (168,161.5) -- cycle ;
\draw   (224,220) .. controls (224,214.48) and (228.48,210) .. (234,210) .. controls (239.52,210) and (244,214.48) .. (244,220) .. controls (244,225.52) and (239.52,230) .. (234,230) .. controls (228.48,230) and (224,225.52) .. (224,220) -- cycle ;
\draw    (244,220) -- (264,203) ;
\draw    (244,220) -- (264,240) ;
\draw  [fill={rgb, 255:red, 0; green, 0; blue, 0 }  ,fill opacity=1 ] (241,220.5) .. controls (241,219.12) and (242.12,218) .. (243.5,218) .. controls (244.88,218) and (246,219.12) .. (246,220.5) .. controls (246,221.88) and (244.88,223) .. (243.5,223) .. controls (242.12,223) and (241,221.88) .. (241,220.5) -- cycle ;
\draw   (157,244) .. controls (157,238.48) and (161.48,234) .. (167,234) .. controls (172.52,234) and (177,238.48) .. (177,244) .. controls (177,249.52) and (172.52,254) .. (167,254) .. controls (161.48,254) and (157,249.52) .. (157,244) -- cycle ;
\draw    (179,244.5) -- (197,244.5) ;
\draw  [fill={rgb, 255:red, 0; green, 0; blue, 0 }  ,fill opacity=1 ] (174,244.5) .. controls (174,243.12) and (175.12,242) .. (176.5,242) .. controls (177.88,242) and (179,243.12) .. (179,244.5) .. controls (179,245.88) and (177.88,247) .. (176.5,247) .. controls (175.12,247) and (174,245.88) .. (174,244.5) -- cycle ;
\draw    (136,244.5) -- (154,244.5) ;
\draw  [fill={rgb, 255:red, 0; green, 0; blue, 0 }  ,fill opacity=1 ] (154,244.5) .. controls (154,243.12) and (155.12,242) .. (156.5,242) .. controls (157.88,242) and (159,243.12) .. (159,244.5) .. controls (159,245.88) and (157.88,247) .. (156.5,247) .. controls (155.12,247) and (154,245.88) .. (154,244.5) -- cycle ;
\draw    (75,225) -- (95,208) ;
\draw    (75,225) -- (95,245) ;
\draw  [fill={rgb, 255:red, 0; green, 0; blue, 0 }  ,fill opacity=1 ] (72,225.5) .. controls (72,224.12) and (73.12,223) .. (74.5,223) .. controls (75.88,223) and (77,224.12) .. (77,225.5) .. controls (77,226.88) and (75.88,228) .. (74.5,228) .. controls (73.12,228) and (72,226.88) .. (72,225.5) -- cycle ;
\draw    (86,84) -- (100,84) ;
\draw    (100,84) -- (120,67) ;
\draw    (100,84) -- (120,104) ;
\draw  [fill={rgb, 255:red, 0; green, 0; blue, 0 }  ,fill opacity=1 ] (83,84.5) .. controls (83,83.12) and (84.12,82) .. (85.5,82) .. controls (86.88,82) and (88,83.12) .. (88,84.5) .. controls (88,85.88) and (86.88,87) .. (85.5,87) .. controls (84.12,87) and (83,85.88) .. (83,84.5) -- cycle ;
\draw  [fill={rgb, 255:red, 0; green, 0; blue, 0 }  ,fill opacity=1 ] (99,84.5) .. controls (99,83.12) and (100.12,82) .. (101.5,82) .. controls (102.88,82) and (104,83.12) .. (104,84.5) .. controls (104,85.88) and (102.88,87) .. (101.5,87) .. controls (100.12,87) and (99,85.88) .. (99,84.5) -- cycle ;
\draw  [color={rgb, 255:red, 255; green, 0; blue, 0 }  ,draw opacity=1 ][fill={rgb, 255:red, 255; green, 0; blue, 0 }  ,fill opacity=1 ] (97,200.5) .. controls (97,199.12) and (98.12,198) .. (99.5,198) .. controls (100.88,198) and (102,199.12) .. (102,200.5) .. controls (102,201.88) and (100.88,203) .. (99.5,203) .. controls (98.12,203) and (97,201.88) .. (97,200.5) -- cycle ;
\draw    (278,219) -- (404,219) ;
\draw [shift={(406,219)}, rotate = 180] [color={rgb, 255:red, 0; green, 0; blue, 0 }  ][line width=0.75]    (10.93,-3.29) .. controls (6.95,-1.4) and (3.31,-0.3) .. (0,0) .. controls (3.31,0.3) and (6.95,1.4) .. (10.93,3.29)   ;
\draw  [draw opacity=0][fill={rgb, 255:red, 244; green, 244; blue, 244 }  ,fill opacity=1 ] (439,176.82) -- (539,276) -- (539.17,177) -- cycle ;
\draw    (439,177) -- (553,177) -- (557,177) ;
\draw [shift={(560,177)}, rotate = 180] [fill={rgb, 255:red, 0; green, 0; blue, 0 }  ][line width=0.08]  [draw opacity=0] (8.93,-4.29) -- (0,0) -- (8.93,4.29) -- cycle    ;
\draw [pattern=_zciasiwn5,pattern size=6pt,pattern thickness=0.75pt,pattern radius=0.75pt, pattern color={rgb, 255:red, 0; green, 0; blue, 0}][line width=0.75]  [dash pattern={on 0.84pt off 2.51pt}]  (439,177) -- (559,297) ;
\draw   (563,176) .. controls (563,170.48) and (567.48,166) .. (573,166) .. controls (578.52,166) and (583,170.48) .. (583,176) .. controls (583,181.52) and (578.52,186) .. (573,186) .. controls (567.48,186) and (563,181.52) .. (563,176) -- cycle ;
\draw    (583,176) -- (603,159) ;
\draw    (583,176) -- (603,196) ;
\draw  [fill={rgb, 255:red, 0; green, 0; blue, 0 }  ,fill opacity=1 ] (580,176.5) .. controls (580,175.12) and (581.12,174) .. (582.5,174) .. controls (583.88,174) and (585,175.12) .. (585,176.5) .. controls (585,177.88) and (583.88,179) .. (582.5,179) .. controls (581.12,179) and (580,177.88) .. (580,176.5) -- cycle ;
\draw   (496,200) .. controls (496,194.48) and (500.48,190) .. (506,190) .. controls (511.52,190) and (516,194.48) .. (516,200) .. controls (516,205.52) and (511.52,210) .. (506,210) .. controls (500.48,210) and (496,205.52) .. (496,200) -- cycle ;
\draw    (518,200.5) -- (536,200.5) ;
\draw  [fill={rgb, 255:red, 0; green, 0; blue, 0 }  ,fill opacity=1 ] (513,200.5) .. controls (513,199.12) and (514.12,198) .. (515.5,198) .. controls (516.88,198) and (518,199.12) .. (518,200.5) .. controls (518,201.88) and (516.88,203) .. (515.5,203) .. controls (514.12,203) and (513,201.88) .. (513,200.5) -- cycle ;
\draw    (475,200.5) -- (493,200.5) ;
\draw  [fill={rgb, 255:red, 0; green, 0; blue, 0 }  ,fill opacity=1 ] (493,200.5) .. controls (493,199.12) and (494.12,198) .. (495.5,198) .. controls (496.88,198) and (498,199.12) .. (498,200.5) .. controls (498,201.88) and (496.88,203) .. (495.5,203) .. controls (494.12,203) and (493,201.88) .. (493,200.5) -- cycle ;
\draw    (418,177) -- (438,160) ;
\draw    (418,177) -- (438,197) ;
\draw  [fill={rgb, 255:red, 0; green, 0; blue, 0 }  ,fill opacity=1 ] (415,177.5) .. controls (415,176.12) and (416.12,175) .. (417.5,175) .. controls (418.88,175) and (420,176.12) .. (420,177.5) .. controls (420,178.88) and (418.88,180) .. (417.5,180) .. controls (416.12,180) and (415,178.88) .. (415,177.5) -- cycle ;
\draw  [color={rgb, 255:red, 255; green, 0; blue, 0 }  ,draw opacity=1 ][fill={rgb, 255:red, 255; green, 0; blue, 0 }  ,fill opacity=1 ] (437,177.5) .. controls (437,176.12) and (438.12,175) .. (439.5,175) .. controls (440.88,175) and (442,176.12) .. (442,177.5) .. controls (442,178.88) and (440.88,180) .. (439.5,180) .. controls (438.12,180) and (437,178.88) .. (437,177.5) -- cycle ;

\draw (67,207.4) node [anchor=north west][inner sep=0.75pt]    {$1$};
\draw (79,65.4) node [anchor=north west][inner sep=0.75pt]    {$1$};
\draw (87,184.4) node [anchor=north west][inner sep=0.75pt]  [color={rgb, 255:red, 255; green, 0; blue, 0 }  ,opacity=1 ]  {$1$};
\draw (410,159.4) node [anchor=north west][inner sep=0.75pt]    {$1$};
\draw (439,158.4) node [anchor=north west][inner sep=0.75pt]  [color={rgb, 255:red, 255; green, 0; blue, 0 }  ,opacity=1 ]  {$0$};
\draw (325,183.4) node [anchor=north west][inner sep=0.75pt]    {$\trop(\pi)$};
\draw (327,238.4) node [anchor=north west][inner sep=0.75pt]    {${\color{red} 1\mapsto 0}$};

\end{tikzpicture}

\caption{The morphism of generalized cone complexes $\trop(\pi)$ from $\mathcal{M}^{\rm trop}_{1,2}$ to $\M_{1,(\epsilon,\epsilon)}^{\rm trop}.$}
\label{fig:hasset}
\end{figure}

Fix $\epsilon << 1$ in $\mathbb{Q}\cap(0,1]$ and consider the genus 1, $n$ markings setting $\M_{1,(\epsilon, \ldots, \epsilon)}$. This space has dimension $n$ but only one boundary divisor $\delta_0$, the divisor of irreducible singular curves. A curve $\Gamma$ in $\M^{\rm trop}_{1,(\epsilon,\ldots,\epsilon)}$ is fixed by specifying the data of an ordering $\alpha\in S_n$ of the markings and $n$ lengths $l_1,\ldots,l_n\in\mathbb{R}_{\geq0}$. The volume of $\Gamma$ is the sum $\text{vol}(\Gamma) = \sum_i l_i$ of the lengths. The resulting metric graph configures the $n$ markings along the circle of circumference $\text{vol}(\Gamma)$ in the order $\alpha$ with the arc lengths $l_i$ between the consecutive markings $\alpha(i)$ and $\alpha(i+1)$. In this way it can be thought of as the configuration space of $n$ points on circles of non-negative circumference.

For $n=2$ this produces a generalized complex given by a single ray corresponding to $\delta_0$ with a folded two dimensional cone glued to it. This is the same cone complex as  $(\M^{\rm ps}_{1,2})^{\rm trop}$. The space $\M_{1,(\epsilon,\epsilon)}$ is the target of a  morphism 
\[\pi:\M_{1,2}\longrightarrow \M_{1,(\epsilon,\epsilon)}\]
given by weighting the two points by $\epsilon$ and stabilizing as a Hassett weighted space. In this case the map $\pi$ is an isomorphism of schemes. However the log structures differ, as in the target space the image of the (interior of the) divisor $\delta_1$ is given trivial log structure, which is to say becomes part of the interior of the moduli space $\M_{1,(\epsilon,\epsilon)}$.

As a consequence, the map $\trop(\pi)$ contracts the ray $\rho_1$ to the origin of the cone complex for weighted tropical curves, as depicted in Figure \ref{fig:hasset}.

\appendix
\section{(by Steffen Marcus and Jonathan Wise)} \label{appendix}

As mentioned in Section~\ref{subsection:trop} the functoriality of Artin fans along $\mathcal{T}:\M_{g,n}\longrightarrow \Mps_{g,n}$ is not assured by the universal property of Artin fans since $\mathcal{T}$ is not strict. In Proposition~\ref{appendixprop:4} we prove a criterion that extends functoriality to a wide array of non-strict cases when combined with the ``patch" given in \cite[Theorem~5.7]{ACMUW}. This criterion is applied to $\mathcal{T}$ in Proposition~\ref{appendixprop:5} to construct the commutative diagram~(\ref{eq:functorial}).

Let $\frak L$ be the algebraic stack whose sections over a scheme $S$ are fine and saturated logarithmic structures on $S$.  Let $\frak L'$ be the algebraic stack whose sections over $S$ are the triples $(P, Q, h)$ where $P$ and $Q$ are fine and saturated logarithmic structures on $S$ and $h : P \to Q$ is a morphism of logarithmic structures.  

\begin{proposition} \label{appendixprop:1}
$\frak L$ and $\frak L'$ are algebraic stacks.
\end{proposition}

\begin{proof}
For $\frak L$, this is \cite[Theorem~1.1]{Olsson}, where $\frak L = \mathcal{L}og$ in the notation of op.\ cit.  For $\frak L'$, observe that $\frak L' = \mathcal{L}og_{\mathcal{L}og}$ in Olsson's notation.
\end{proof}

\begin{proposition} \label{appendixprop:2}
The projection $r : \frak L' \to \frak L$ sending $(P \to Q)$ to $Q$ is \'etale.
\end{proposition}

\begin{proof}
We may use the infinitesimal criterion.  Given a morphism of logarithmic structures on a scheme $S$, an infinitesimal extension $S'$ of $S$, and an extension $Q'$ of $Q$, the characteristic monoid $\bar P$ of $P$, and the homomorphism $\bar P \to \bar Q$, extend uniquely to $S'$ since the \'etale sites of $S$ and of $S'$ are the same.  Then an extension $P'$ of $P$ is uniquely determined from $Q'$ by pulling back $Q'$ along the homomorphism $\bar P' \to \bar Q'$.
\end{proof}

\begin{proposition} \label{appendixprop:3}
Let $r : \frak L' \to \frak L$ be the morphism sending $(P \to Q)$ to $Q$.  Let $U$ be the substack of $\frak L'$ where $P \to Q$ is injective.  Then $U \subset \frak L'$ is open and $r \big|_U$ is representable by algebraic spaces.
\end{proposition}

\begin{proof}
For the openness:  $h : P \to Q$ is injective if and only if the associated map of characteristic monoids is injective.  Since $\bar P$ has local charts by finitely generated monoids, $\bar h : \bar P \to \bar Q$ is injective at a point if and only if it is injective in all sufficiently small neighborhoods of the point.  Therefore $U$ is open.

For the representability: an automorphism of $(P \to Q)$ is determined by its action on $Q$ when $P \to Q$ is injective.
\end{proof}

\begin{proposition} \label{appendixprop:4}
Let $X$ and $Y$ be logarithmically flat algebraic stacks.  Let $\frak X = \pi_0(X/\frak L)$ and $\frak Y = \pi_0(Y/\frak L)$ be the Artin fans of $X$ and $Y$.  Let $\frak X' = \pi_0(X/\frak L')$ be the relative Artin fan of $X$ over $\frak Y$.  If $\overline{M}_{Y,f(x)} \to \overline{M
}_{X,x}$ is injective for all geometric points $x$ of $X$ then $\frak X' \to \frak X$ is an isomorphism.
\end{proposition}

\begin{proof}
Since $X \to \frak X'$ is surjective with connected fibers, we may identify $\pi_0(X/\frak L)$ with $\pi_0(\frak X'/\frak L)$ where $\frak X' \to \frak L$ is the composition of $\frak X' \to \frak L'$ and the projection $\frak L' \to \frak L$ sending $(P \to Q)$ to $Q$.  But $\frak X' \to \frak L'$ factors through $U$, in the notation of Proposition~\ref{appendixprop:3}, and proposition~\ref{appendixprop:3} shows that $\frak X' \to U \to \frak L$ is representable.  Since $\frak X' \to \frak L$ is also \'etale by Proposition~\ref{appendixprop:2}, it follows that $\pi_0(\frak X'/\frak L) = \frak X'$, as required.
\end{proof}

Note that $\M_{g,n}$ is logarithmically smooth, and thus logarithmically flat. The same is true for $\Mps_{g,n}$.

\begin{proposition} \label{prop:logsmooth}
$\Mps_{g,n}$ is logarithmically smooth.
\end{proposition}
\begin{proof}

The moduli space $\Mps_{g,n}$ of pseudostable curves is locally of finite presentation over $k$.  By definition \cite[Section~3.3]{KatoK}, we must show that given a logarithmic square-zero extension $T\longrightarrow T'$ defined by an ideal $I$ and a pseudostable curve $\pi:C\longrightarrow T$, there exists a compatible first-order deformation $C'\longrightarrow T'$ extending $C$ with ideal $I$: 
\[
\xymatrix{
C \ar[r] \ar[d]_\pi& C' \ar[d]\\
T \ar[r] & T'.}
\]
Denote by $\mathbb{L}_{C/T}$ the logarithmic cotangent complex of $C$ relative to $T$. The problem of finding this extension is obstructed by a class in ${\rm Ext}^2(\mathbb{L}_{C/T},I)$ and, if this class is zero, then isomorphism classes of such extensions form a torsor under ${\rm Ext}^1(\mathbb{L}_{C/T},I)$. Thus it is enough to prove that ${\rm Ext}^2(\mathbb{L}_{C/T},I)=0$. Using the local-to-global spectral sequence for ${\rm Ext}$, we have 
\[H^p(C,\mathcal{E}xt^q(\mathbb{L}_{C/T},I)) \implies {\rm Ext}^{p+q}(\mathbb{L}_{C/T},I).\]

For $p+q=2$ there are three contributions to consider. First, since the map $C \longrightarrow \mathcal{L}og(T)$ is lci it follows that $\mathbb{L}_{C/T}$ is perfect in 2 degrees, so \[H^0(C,\mathcal{E}xt^2(\mathbb{L}_{C/T},I))=0.\]
Since $C$ is logarithmically smooth away from cuspoidal points, $\mathbb{L}_{C/T}$ is locally free in degree 0 away from cuspoidal points, so  $\mathcal{E}xt^1(\mathbb{L}_{C/T},I)$ is concentrated on a closed subscheme that is finite over $T$, giving 
\[H^1(C,\mathcal{E}xt^1(\mathbb{L}_{C/T},I))=0.\]
Finally, 
\[H^2(C,\mathcal{E}xt^0(\mathbb{L}_{C/T},I))=0\]
because C is a curve. Thus ${\rm Ext}^2(\mathbb{L}_{C/T},I)=0$.

    
\end{proof}

Applying the usual functoriality ``patch" \cite[Theorem~5.7]{ACMUW} for Artin fans to $\mathcal{T}$ in combination with Proposition~\ref{appendixprop:4} produces the commutative diagram~(\ref{eq:functorial}). 

\begin{proposition}\label{appendixprop:5}
There is an initial commutative diagram
\[\xymatrix{
\M_{g,n} \ar[r]^{\mathcal{T}} \ar[d]_{\alpha}& \Mps_{g,n} \ar[d]^{\alpha^{\rm ps}}\\
\mathcal{A}_{\M_{g,n}} \ar[r] & \mathcal{A}_{\Mps_{g,n}} 
}\] where the vertical arrows are strict.
\end{proposition}
\begin{proof}
We will show the hypotheses of Proposition~~\ref{appendixprop:4} are satisfied. By \cite[Theorem~5.7]{ACMUW} there is an initial commutative diagram
\[\xymatrix{
\M_{g,n} \ar[d]_{\mathcal{T}} \ar[r]&  
 \pi_0(\M_{g,n}/\frak L') \ar[r]\ar[d] &\frak L'\ar[d]\\
\Mps_{g,n}  \ar[r]^{\alpha^{\rm ps}}  & \mathcal{A}_{\Mps_{g,n}}\ar[r]&\frak L&
}\] 
in which the horizontal arrows are strict. Given any geometric point $x$ of $\M_{g,n}$, if $x\not\in\delta_1$ then $\mathcal{T}(x)=x$ and the map $\mathcal{M}_{\Mps_{g,n},\mathcal{T}(x)}\longrightarrow\mathcal{M}_{\M_{g,n},x}$ is the identity, giving an injective map on characteristics. If $x\in\delta_1$ then each cusp on the curve $\mathcal{T}(x)$ determines an element $a^{\ps}$ of the log structure on $\mathcal{T}(x)$ corresponding to the divisor $\delta_0^{\ps}$ and generating a factor $\mathbb{N}\subset \overline{\mathcal{M}}_{\Mps_{g,n},\mathcal{T}(x)}$ of the characteristic monoid. Since away from such factors the map on characteristic monoid is the identity, it is enough to show that this map in injective on these factors.

As described in the proof of Theorem~\ref{thm:PLtropmap}, $a^{\ps}$ maps via $\mathcal{M}_{\Mps_{g,n},\mathcal{T}(x)}\longrightarrow\mathcal{M}_{\M_{g,n},x}$ as \[a^{\ps}\mapsto a+12b\]
where $a$ and $b$ are elements of the log structure on $x$ corresponding to the divisors $\delta_0$ and $\delta_1$ respectively. While it may be that $a$ is invertible in $\mathcal{M}_{\M_{g,n},x}$ if the elliptic tail contributing the cusp does not have an irreducible node, $b$ can not be invertible. Thus the map on characteristics is injective and by Proposition~\ref{appendixprop:4} we have $\mathcal{A}_{\M_{g,n}}\cong \pi_0(\M_{g,n}/\frak L')$ giving the desired diagram.
\end{proof}

\bibliographystyle{amsalpha}
\bibliography{lambdag}

\end{document}